\documentclass[11pt]{article}
\usepackage[T1]{fontenc}
\usepackage{lmodern}
\usepackage[margin=1in]{geometry}

\usepackage{amsmath,amssymb,amsthm,mathtools,mathrsfs}
\usepackage{microtype}
\usepackage{enumitem}
\usepackage{booktabs,graphicx,array}
\usepackage{hyperref}
\usepackage{xcolor}
\usepackage{multirow}
\usepackage{subfigure}

\numberwithin{equation}{section}
\newtheorem{theorem}{Theorem}[section]
\newtheorem{lemma}[theorem]{Lemma}
\newtheorem{proposition}[theorem]{Proposition}

\newtheorem{assumption}[theorem]{Assumption}
\theoremstyle{remark}
\newtheorem{remark}[theorem]{Remark}

\newtheorem{problem}{Problem}

\newcommand{\norm}[1]{\left\lVert #1\right\rVert}

\title{Inverse heat source problem from terminal data on a set of low Hausdorff dimension}

\author{
Kai Yu \thanks{School of Mathematics and Statistics, Ningbo University, Ningbo 315211, China. yukaimailbox@163.com} \and Zhiyuan Li \thanks{School of Mathematics and Statistics, Ningbo University, Ningbo 315211, China.  lizhiyuan@nbu.edu.cn} \and Zhidong Zhang \thanks{School of Mathematics (Zhuhai), Sun Yat-sen University, Zhuhai 519082, China. Guangdong Province Key Laboratory of Computational Science, Sun Yat-sen University, Guangzhou 510000, China. zhangzhidong@mail.sysu.edu.cn}
}
\date{}

\begin{document}
\maketitle

\begin{abstract}

We study the recovery of a spatial source in the heat equation from terminal observations on a compact interior set $E$ with positive $\beta$ dimensional Hausdorff content for some $\beta>d-1$. The prescribed temporal forcing vanishes on a nontrivial interval immediately preceding the observation time. Under a spectral Sobolev prior and a quantitative noncancellation condition on the terminal spectral multipliers, we establish conditional logarithmic stability. For scattered observations on such sets, we derive a sampling inequality under a quasi uniformity assumption. Under explicit sampling and independent noise assumptions, we prove a logarithmic convergence rate in probability for the corresponding Tikhonov reconstructions with a data driven parameter choice. Two numerical experiments using a finite candidate approximation of a fixed Hausdorff set assess discrete spectral sampling stability and illustrate the effects of the waiting interval and measurement noise on reconstruction. 

\medskip

\noindent\textbf{Keywords.} inverse source problem; terminal observation; Hausdorff content; scattered point measurement; Tikhonov regularization.

\medskip
\noindent\textbf{MSC 2020.} 35R30, 35K05, 65J20.

\end{abstract}

\section{Introduction}

\subsection{Model and motivation}

Inverse source problems for diffusion equations seek to recover an unknown
forcing term from indirect measurements of the corresponding state. They are
ill posed because diffusion strongly damps high spatial frequencies. The
problem becomes more challenging when observations are restricted to a
geometrically thin subset of the domain. We ask whether a spatial source can be recovered, and with
what stability, from a terminal measurement on a compact set with positive
$\beta$-dimensional Hausdorff content for some $\beta\in(d-1,d]$. The set may have empty interior and zero $d$-dimensional Lebesgue measure.

More precisely, let $\Omega\subset\mathbb R^d$, $d\ge1$, be a bounded connected
$C^{1,1}$ domain. We consider the heat
equation
\begin{equation}
\label{eq:model}
\begin{cases}
\partial_tu-\Delta u=\rho(t)f(x),
&(x,t)\in\Omega\times(0,T),\\
u(x,t)=0,
&(x,t)\in\partial\Omega\times(0,T),\\
u(x,0)=0,
&x\in\Omega.
\end{cases}
\end{equation}
The source term $\rho(t)f(x)$ has a separable structure, where
$\rho(t)$ describes the temporal variation of the source intensity and $f(x)$ characterizes its spatial distribution.

The Dirichlet Laplacian $A:=-\Delta$ is positive and self-adjoint, and the
heat semigroup $e^{-tA}$ gives a spectral representation of the terminal
source-to-state map, which is also useful for the numerical approximation
below \cite{Davies89,Pazy83}. A key assumption is that the temporal forcing
vanishes on a nontrivial interval immediately preceding the observation time:
\begin{equation}
\label{eq:switch-off}
\rho(t)=0
\quad\text{for a.e. }t\in(T_0,T),\qquad \tau_0:=T-T_0>0.
\end{equation}
The temporal forcing vanishes throughout $(T_0,T)$, so that
\[
u(T)=e^{-\tau_0A}u(T_0).
\]
In each eigenmode, this evolution contributes the factor
$e^{-\tau_0\lambda_k}$. The resulting additional smoothing is central to
the analysis: it must offset the cost of propagating information from the
thin observation set to the whole domain.

The observation set may be much smaller than an open subset of $\Omega$.
For $0<\beta\le d$ and $E\subset\mathbb R^d$, its $\beta$-dimensional
Hausdorff content is
\begin{equation}
\label{eq:content}
\mathcal H_\infty^\beta(E)
:=\inf\Bigl\{\sum_j r_j^\beta:
E\subset\bigcup_jB(x_j,r_j)\Bigr\}.
\end{equation}
We assume that $E\Subset\Omega$ is compact and
\begin{equation}
\label{eq:haus}
\mathcal H_\infty^\beta(E)>0,
\qquad d-1<\beta\le d.
\end{equation}
The positive waiting time gives a continuous representative of the terminal
state on $E$ for every $f\in L^2(\Omega)$, so the observation
$u(\cdot,T)|_E$ is well defined. Therefore, we formulate the inverse source problem as follows.

\begin{problem}
Let $\rho$ be known and satisfy \eqref{eq:switch-off}, and let
$E\Subset\Omega$ be a compact set satisfying \eqref{eq:haus}.
Given the terminal observation \(u(\cdot,T)|_E, \)
where $u$ solves \eqref{eq:model}, determine the unknown spatial
source $f(x)$.
\end{problem}

This problem is motivated by source identification in atmospheric
science, where pollutant release characteristics are inferred from
concentration measurements at monitoring sites
\cite{Rao07,WangAtmospheric21}. In practice, the number and spatial
distribution of monitoring sites may be limited by measurement costs
and geographical constraints, making observations over an entire
open region unavailable. This raises a natural question: can a
spatial source be stably recovered from terminal data restricted
to a spatially thin set, potentially with empty interior and zero
$d$-dimensional Lebesgue measure? The theoretical results developed
in this work may provide a mathematical foundation for addressing
source identification problems under such limited observation
conditions.

\subsection{Relevant literature and the underlying difficulty}

Inverse source problems for classical parabolic equations have been studied extensively, particularly with regard to uniqueness and conditional stability under suitable a priori assumptions. Carleman estimates are an important tool for deriving stability estimates; see, for example, \cite{Yamamoto93,ChoulliYamamoto04,Yamamoto09,ImanuvilovYamamoto98}. Other approaches include the analysis of data compatibility and conditional stability, as well as operator theoretic formulations for abstract evolution equations \cite{Li06,PrilepkoTikhonov95,JaradatEtAl17}. Earlier studies also addressed inverse source recovery from overspecified boundary data and final time observations \cite{Cannon68,Isakov91}, while variational methods based on adjoint equations have been used to establish conditional stability for diffusion inverse source problems \cite{SunLiuLi17}. Related questions of uniqueness and numerical reconstruction have also been investigated for coupled diffusion models arising in fluorescence optical tomography \cite{SunZhang22}. Another closely related problem is the backward heat problem, which seeks to recover an initial temperature from observations at later times and is severely ill posed because heat evolution exponentially damps high spatial frequencies; see \cite{Liu02,FuXiongQian07,LiYamamotoZou09,CasagrandeDelSantoPrizzi22}.

The geometry of observation sets is also relevant to
quantitative observability. While classical observability estimates are often formulated for
observations on interior open sets or prescribed portions of the
boundary, quantitative estimates for the propagation of smallness
and spectral inequalities can also be established for certain sets
with Hausdorff dimension strictly greater than $d-1$
\cite{BurqMoyano23,Green25}. Such techniques have been applied to
inverse heat-source problems \cite{LiLiZhou26} and to the recovery
of initial data for fractional heat equations from observations
distributed in time on thin sets \cite{YuLi26}.
These developments provide a framework for studying inverse problems
with observations on sets of small Hausdorff dimension, although the
precise stability estimates depend on the equation and the observation
configuration.

Numerical reconstruction of exponentially ill-posed inverse problems commonly employs spectral filtering and Tikhonov regularization \cite{EnglHankeNeubauer96,Hansen10}. For the inverse source problem for the heat equation, Chen, Zhang, and Zou studied quasi-uniform scattered measurements and stochastic Tikhonov regularization \cite{ChenZhangZou22}. Related work has addressed data-driven model reduction for parabolic inverse source problems \cite{WangZhangZhang231} and stochastic convergence of regularized solutions for backward heat conduction problems \cite{WangZhangZhang23}. Probabilistic error estimates for reconstruction from discrete noisy data have also been obtained for nonlinear inverse source problems in coupled diffusion equations, together with stability analysis based on terminal observations \cite{SunZhangZhang25}. From the perspective of discrete observation, finite-dimensional sampling inequalities provide a means of controlling functions from point measurements \cite{LimonovaMalykhinTemlyakov24,KriegPozharskaUllrich26}, while spectral sampling and Fekete-point constructions offer related approaches to stable discretization \cite{OrtegaPridhnani12,SommarivaVianello09}. Data-driven parameter choice is another aspect of regularization, and fixed-point updates have been studied by Cen, Li, and Zhang \cite{CenLiZhang25}. These works provide relevant analytical and computational tools, but establishing stability and convergence for the recovery of a spatial source from terminal heat observations on thin sets requires addressing the interaction between discrete sampling and the geometry of the observation set.

Against this background, we study the recovery of a spatial source from
terminal data on a compact set $E\Subset\Omega$ satisfying
$\mathcal H_\infty^\beta(E)>0$ for some $\beta\in(d-1,d]$.
The analysis requires a quantitative noncancellation condition on the
terminal multipliers, since a sign changing temporal profile $\rho$ may
cause cancellation at individual eigenfrequencies. The waiting interval
during which the forcing vanishes contributes the exponential factor
$e^{-\tau_0\Lambda}$, which dominates the spectral observation cost
$e^{C\sqrt{\Lambda}}$ at high frequencies. This comparison is central to
the conditional stability and discrete reconstruction results established
below.

\subsection{Main results and contributions}

We now formulate precisely the spectral noncancellation hypothesis used in
the stability analysis.

\begin{assumption}
\label{ass:rho}
Let $\rho\in L^\infty(0,T)$ satisfy \eqref{eq:switch-off}, and let
$0<\lambda_1\le\lambda_2\le\cdots$ be the Dirichlet eigenvalues of $A$,
counted with multiplicity. For each $k\ge1$, define the terminal multiplier
\[
b_k:=\int_0^{T_0}\rho(t)e^{-(T-t)\lambda_k}\,dt.
\]
There exist constants $a\ge1$ and $c_b>0$ such that
\begin{equation*}
|b_k|\ge c_b\lambda_k^{-a}e^{-\tau_0\lambda_k}
\qquad\text{for every }k\ge1.
\end{equation*}
\end{assumption}

The elementary upper estimate
\[
|b_k|
\le \|\rho\|_{L^\infty(0,T_0)}
e^{-\tau_0\lambda_k}
\int_0^{T_0}e^{-(T_0-t)\lambda_k}\,dt
\le \|\rho\|_{L^\infty(0,T_0)}
\lambda_k^{-1}e^{-\tau_0\lambda_k}
\]
follows directly from $\rho\in L^\infty(0,T_0)$. Combining this estimate
with Assumption~\ref{ass:rho}, we obtain the two-sided multiplier bound
\begin{equation}
\label{eq:multiplier-two-sided}
c_b\lambda_k^{-a}e^{-\tau_0\lambda_k}
\le |b_k|
\le \|\rho\|_{L^\infty(0,T_0)}
\lambda_k^{-1}e^{-\tau_0\lambda_k},
\qquad k\ge1.
\end{equation}
Assumption~\ref{ass:rho} excludes complete or super-polynomial cancellation
of the temporal forcing along the Dirichlet spectrum, relative to the
exponential decay generated by the waiting time. Remark~\ref{rem:profile-conditions}
below gives concrete sufficient conditions on $\rho$ ensuring this property.

\begin{remark}
\label{rem:profile-conditions}
Let $a\ge1$ and let $\rho\in L^\infty(0,T)$ satisfy
\eqref{eq:switch-off}. Each of the following conditions implies
Assumption~\ref{ass:rho} with this exponent $a$.

\begin{enumerate}[label=\textup{(\roman*)}]
\item \textit{Nonnegative temporal profiles.}
There exist $\eta\in(0,T_0]$ and $c_*>0$ such that
$\rho\ge0$ a.e. on $(0,T_0)$ and
$\rho(t)\ge c_*(T_0-t)^{a-1}$ a.e. on $(T_0-\eta,T_0)$.

\item \textit{Temporal profiles allowing sign changes.}
There exist $\eta\in(0,T_0]$, $\alpha>0$, and $\kappa\ne0$
such that $\rho(T_0-r)=\kappa r^{a-1}+O(r^{a-1+\alpha})$
for a.e. $r\in(0,\eta)$, and $b_k\ne0$ for every $k\ge1$.
\end{enumerate}

To verify these claims, write $\mu=\lambda_k$. In case \textup{(i)}, nonnegativity gives
\[
\mu^a e^{\tau_0\mu}b_k
\ge c_*\mu^a\int_0^\eta r^{a-1}e^{-\mu r}\,dr
=c_*\int_0^{\eta\mu}v^{a-1}e^{-v}\,dv
\ge c_*\int_0^{\eta\lambda_1}v^{a-1}e^{-v}\,dv>0.
\]
For \textup{(ii)}, split the integral defining $e^{\tau_0\mu}b_k$
at $r=\eta$. Boundedness of $\rho$ and the endpoint expansion yield
\[
\mu^a e^{\tau_0\mu}b_k
=\kappa\int_0^{\eta\mu}v^{a-1}e^{-v}\,dv
+O(\mu^{-\alpha})+O(\mu^{a-1}e^{-\eta\mu})
\longrightarrow\kappa\Gamma(a)\ne0
\qquad(\mu\to\infty).
\]
Thus the required bound holds for all sufficiently large $k$.
It holds for the remaining finitely many $k$ after reducing the positive
constant, since each corresponding $b_k$ is nonzero.
\end{remark}

\begin{theorem}[Conditional logarithmic stability]
\label{thm:log}
Assume \eqref{eq:haus} and
Assumption~\ref{ass:rho}, let
$\sigma>0$, and let $f_1,f_2\in D(A^\sigma)$ satisfy
\[
\|A^\sigma(f_1-f_2)\|_{L^2(\Omega)}\le M
\]
for some $M>0$. Set
\[
\delta_T:=\sup_{x\in E}|u_1(x,T)-u_2(x,T)|.
\]
Then
\begin{equation}
\label{eq:log-main}
\|f_1-f_2\|_{L^2(\Omega)}
\le
CM\left[\log\left(e+\frac{M}{\delta_T}\right)\right]^{-\sigma},
\end{equation}
where $C>0$ may depend on $\Omega$, $E$, $\beta$, $\sigma$,
$T_0$, $\tau_0$, $a$, $c_b$, and $\|\rho\|_{L^\infty(0,T)}$,
but is independent of $f_1$, $f_2$, $M$, and $\delta_T$. The right-hand side is interpreted as zero when
$\delta_T=0$.
\end{theorem}

In summary, the main contributions of the paper are as follows.
\begin{enumerate}[label=\textup{(\roman*)}]
\item
We combine the Hausdorff spectral inequality with the smoothing
induced by the positive waiting time to establish conditional
logarithmic stability for the recovery of the spatial source under
the spectral noncancellation condition and the prior
$f\in D(A^\sigma)$.

\item
For scattered observations, we derive a frequency-dependent
quasi-uniform sampling inequality and a terminal high-frequency
remainder estimate. Under explicit quasi-uniformity and independent-noise
assumptions, the reconstructions obtained using the data-driven
parameter rule converge in probability at the logarithmic rate
$(\log n)^{-\sigma}$.

\item
We conduct two numerical experiments using a finite candidate pool approximating $E_3$.
Experiment A investigates the discrete sampling
stability through the one-sided sampling constants and sampling-matrix
condition numbers for varying spectral dimensions and numbers of sensors.
Experiment B illustrates the effects of the source-free waiting time
and measurement noise on the reconstruction of smooth and nonsmooth
sources.
\end{enumerate}

The remainder of the paper is organized as follows.
Section~\ref{sec:forward} presents the forward spectral representation
and studies the terminal multipliers.
Section~\ref{sec:hsi} establishes the interpolation estimate on the
Hausdorff observation set.
Section~\ref{sec:multiplier} proves the logarithmic stability estimate
and the loss of injectivity when the source remains
active at the observation time.
Section~\ref{sec:numerical} develops the sampling and regularization
theory for scattered observations.
Section~\ref{sec:experiments} presents the two numerical experiments.
Finally, Section~\ref{sec:conclusion} concludes the paper.

\section{Forward analysis and temporal multipliers}
\label{sec:forward}

Throughout the paper, unless otherwise stated, constants may depend on the fixed parameters
$\Omega$, $E$, $\beta$, $T$, $T_0$, $\|\rho\|_{L^\infty(0,T)}$, $a$, and $c_b$,
and, where appropriate, on $\sigma$ or auxiliary interpolation parameters.
They are independent of the unknown source, the a priori bound $M$, and the
measurement error.

Let $A=-\Delta$ be the Dirichlet Laplacian in $L^2(\Omega)$.
There exist eigenpairs
\begin{equation*}
A\phi_k=\lambda_k\phi_k,
\qquad
0<\lambda_1\le\lambda_2\le\cdots\to\infty,
\end{equation*}
where $(\phi_k)_{k\ge1}$ is an orthonormal basis of $L^2(\Omega)$. For $r\ge0$,
\begin{equation*}
A^rv=\sum_{k\ge1}\lambda_k^r v_k\phi_k,
\qquad
v_k=(v,\phi_k),
\end{equation*}
with domain
\begin{equation*}
D(A^r)=\Bigl\{v\in L^2(\Omega):
\sum_{k\ge1}\lambda_k^{2r}|v_k|^2<\infty\Bigr\}.
\end{equation*}
In particular $A$ is positive self-adjoint, and $-A$ generates the contraction semigroup
\begin{equation*}
e^{-tA}v=\sum_{k\ge1}e^{-t\lambda_k}v_k\phi_k.
\end{equation*}

For $f\in L^2(\Omega)$ and $\rho\in L^1(0,T)$, the mild solution of \eqref{eq:model} is
\begin{equation}
\label{eq:Duhamel}
u(t)=\int_0^t\rho(r)e^{-(t-r)A}f\,dr.
\end{equation}
Equivalently,
\begin{equation}
\label{eq:spectral-u}
u(t)=\sum_{k\ge1}\int_0^t\rho(r)e^{-\lambda_k(t-r)}\,dr f_k\phi_k.
\end{equation}
The mild solution is well defined and unique in $C([0,T];L^2(\Omega))$ by standard semigroup theory \cite{Pazy83}.

We first recall the continuity and $L^\infty$ bounds of the
Dirichlet heat semigroup, which ensure that pointwise observations
on the compact set $E$ are well defined and bounded.

\begin{lemma}
\label{lem:smoothing}
\label{lem:interior-smooth}
For every $t>0$ and $v\in L^2(\Omega)$, the function
$e^{-tA}v$ has a continuous representative in $\Omega$.
Moreover, for each fixed $t_0>0$,
\begin{equation}
\label{eq:L2Linf}
\|e^{-tA}v\|_{L^\infty(\Omega)}
\le C(t_0)\|v\|_{L^2(\Omega)},
\qquad t\ge t_0.
\end{equation}
In particular, for every compact set $E\Subset\Omega$,
the restriction $e^{-tA}v|_E$ is continuous and satisfies
\[
\sup_{x\in E}|e^{-tA}v(x)|
\le C(t_0)\|v\|_{L^2(\Omega)},
\qquad t\ge t_0.
\]
\end{lemma}

\begin{proof}
For every $m\ge0$, the spectral theorem gives
\[
\begin{aligned}
\|A^m e^{-tA}\|_{L^2(\Omega)\to L^2(\Omega)}
=\sup_{\lambda\in\sigma(A)}\lambda^m e^{-t\lambda}\le t^{-m}\sup_{r>0}r^m e^{-r}
=C_m t^{-m}.
\end{aligned}
\]
In particular, $e^{-tA}v\in D(A^m)$ for every $t>0$
and $v\in L^2(\Omega)$.

The standard ultracontractivity estimate for the Dirichlet
heat semigroup (see \cite{Davies89,Ouhabaz05}) gives
\[
\|e^{-tA}v\|_{L^\infty(\Omega)}
\le Ct^{-d/4}\|v\|_{L^2(\Omega)},
\qquad t>0.
\]
Thus, for $t\ge t_0>0$,
\[
\|e^{-tA}v\|_{L^\infty(\Omega)}
\le Ct_0^{-d/4}\|v\|_{L^2(\Omega)},
\]
which proves \eqref{eq:L2Linf}.

It remains to establish continuity. Fix a compact set
$K\Subset\Omega$ and open sets $U,V$ such that
$K\Subset V\Subset U\Subset\Omega$.
Choose an integer $\ell\ge1$ with $2\ell>d/2$.
Interior elliptic regularity, iterated $\ell$ times
(see \cite{GilbargTrudinger01}), yields
\[
\|w\|_{H^{2\ell}(V)}
\le C\sum_{j=0}^{\ell}\|A^j w\|_{L^2(U)},
\qquad w\in D(A^\ell).
\]
Applying this estimate to $w=e^{-tA}v$ and using the
spectral estimate above, we obtain
\[
\|e^{-tA}v\|_{H^{2\ell}(V)}
\le C\sum_{j=0}^{\ell}t^{-j}\|v\|_{L^2(\Omega)}.
\]
Since $2\ell>d/2$, the local Sobolev embedding gives
a continuous representative near $K$. These representatives
agree on overlapping neighborhoods, since they are continuous
and agree almost everywhere. As $K\Subset\Omega$ is arbitrary,
they define a continuous representative in $\Omega$.

For this continuous representative, the pointwise absolute
value is bounded by the essential supremum on $\Omega$.
Hence \eqref{eq:L2Linf} implies
\[
\sup_{x\in E}|e^{-tA}v(x)|
\le C(t_0)\|v\|_{L^2(\Omega)},
\qquad t\ge t_0.
\]
This completes the proof.
\end{proof}

\section{Hausdorff spectral inequality and terminal interpolation}
\label{sec:hsi}

\subsection{Hausdorff spectral inequality}

Recall that for $0<\beta\le d$ and $E\subset\mathbb R^d$, the
$\beta$-dimensional Hausdorff content is defined by \eqref{eq:content}; see
\cite{Falconer90,Mattila95} for background. In particular, $E$ may have zero
$d$-dimensional Lebesgue measure.

Let
\[
\Pi_\Lambda v:=\sum_{\lambda_k\le\Lambda}v_k\phi_k .
\]
The central geometric input is the following consequence of
\cite{Green25}. Related propagation-of-smallness results appear in
\cite{LogunovMalinnikova18,Malinnikova04}.

\begin{proposition}[Hausdorff spectral inequality]
\label{prop:HSI}
Assume \eqref{eq:haus}. There exist constants $C_1\ge1$ and
$\Lambda_0\ge1$, depending only on $\Omega$, $E$, and $\beta$, such that
\begin{equation}
\label{eq:HSI}
\norm{\Pi_\Lambda v}_{L^2(\Omega)}
\le C_1 e^{C_1\sqrt\Lambda}
\sup_{x\in E}|\Pi_\Lambda v(x)|
\end{equation}
for every $v\in L^2(\Omega)$ and every $\Lambda\ge\Lambda_0$.
\end{proposition}

Estimate \eqref{eq:HSI} is Theorem~2.1 of \cite{Green25}; the hypothesis \eqref{eq:haus} corresponds to the positive $(d-1+\delta)$-dimensional Hausdorff-content assumption there, with $\delta=\beta-(d-1)>0$. The threshold $d-1$ is natural for this theorem; it is also the dimension at which nodal hypersurfaces obstruct completely general observation statements. The finer log-Hausdorff theory of \cite{HuangWangWang26} suggests possible critical refinements, but we do not use them here.

\subsection{Terminal interpolation from thin sets}

The following proposition bridges the terminal trace on $E$ and the global final state. It combines the Hausdorff spectral inequality
with the exponential decay generated by the positive waiting time. The decisive comparison is between the observation cost
$e^{C\sqrt\Lambda}$ and the terminal damping
$e^{-c\Lambda}$.

\begin{proposition}
\label{prop:terminal-interp}
Assume \eqref{eq:haus} and let $\tau_0>0$. For every
$\theta\in(0,1)$ there exists $C_\theta\ge1$ such that every
\[
y=e^{-\tau_0A}z,
\qquad z\in L^2(\Omega),
\]
satisfies
\begin{equation}
\label{eq:terminal-interp}
\norm y_{L^2(\Omega)}
\le
C_\theta
\norm z_{L^2(\Omega)}^{1-\theta}
\sup_{x\in E}|y|^\theta .
\end{equation}
\end{proposition}

\begin{proof}
If $\norm z_{L^2(\Omega)}=0$, there is nothing to prove, we consider $\norm z_{L^2(\Omega)}>0$. For $\Lambda\ge\Lambda_0$, we first estimate the high-frequency part generated by the terminal waiting interval. Since \(y=e^{-\tau_0A}z,\) the spectral representation gives
\[
(I-\Pi_\Lambda)y
=
\sum_{\lambda_k>\Lambda}
e^{-\tau_0\lambda_k}z_k\phi_k .
\]
Therefore,
\begin{equation}
\label{eq:high-L2-direct}
\norm{(I-\Pi_\Lambda)y}_{L^2(\Omega)}
\le
e^{-\tau_0\Lambda}\norm z_{L^2(\Omega)} .
\end{equation}

To obtain the uniform norm estimate, split the semigroup at the midpoint of
the waiting interval:
\[
(I-\Pi_\Lambda)y
=
e^{-\frac{\tau_0}{2}A}
(I-\Pi_\Lambda)
e^{-\frac{\tau_0}{2}A}z .
\]
By Lemma~\ref{lem:smoothing},
\[
\norm{
e^{-\frac{\tau_0}{2}A}v
}_{L^\infty(\Omega)}
\le C_{\tau_0}\norm v_{L^2(\Omega)} .
\]
On the other hand,
\[
\norm{
(I-\Pi_\Lambda)
e^{-\frac{\tau_0}{2}A}z
}_{L^2(\Omega)}
\le
e^{-\frac{\tau_0}{2}\Lambda}\norm z_{L^2(\Omega)} .
\]
Consequently,
\begin{equation}
\label{eq:high-Linf-direct}
\norm{(I-\Pi_\Lambda)y}_{L^\infty(\Omega)}
\le
C e^{-\frac{\tau_0}{2}\Lambda}\norm z_{L^2(\Omega)} .
\end{equation}

Using Proposition~\ref{prop:HSI}, we have
\[
\|\Pi_\Lambda y\|_{L^2(\Omega)}
\le
C_1 e^{C_1\sqrt\Lambda}
\sup_{x\in E}|\Pi_\Lambda y|.
\]
Since
\[
\Pi_\Lambda y
=
y-(I-\Pi_\Lambda)y ,
\]
we obtain
\[
\sup_{x\in E}|\Pi_\Lambda y|
\le \sup_{x\in E}|y|+
\|(I-\Pi_\Lambda)y\|_{L^\infty(\Omega)} .
\]
Hence, by \eqref{eq:high-Linf-direct},
\[
\|\Pi_\Lambda y\|_{L^2(\Omega)}
\le
C_1 e^{C_1\sqrt\Lambda}
\left(
\sup_{x\in E}|y|+
C e^{-\frac{\tau_0}{2}\Lambda}\norm z_{L^2(\Omega)}
\right).
\]
Combining this estimate with \eqref{eq:high-L2-direct}, we obtain
\begin{align*}
\norm y_{L^2(\Omega)}
&\le
\norm{\Pi_\Lambda y}_{L^2(\Omega)}
+
\norm{(I-\Pi_\Lambda)y}_{L^2(\Omega)}
\nonumber\\
&\le
C e^{C_1\sqrt\Lambda}\sup_{x\in E}|y|
+
C\norm z_{L^2(\Omega)}
e^{C_1\sqrt\Lambda-\frac{\tau_0}{2}\Lambda}
+
\norm z_{L^2(\Omega)} e^{-\tau_0\Lambda}.
\end{align*}

Fix $\theta\in(0,1)$. Since $\sqrt\Lambda=o(\Lambda)$, there exists $C_\theta'$ such that
\[
C_1\sqrt\Lambda
\le
\frac{1-\theta}{2}\tau_0\Lambda+C_\theta',
\qquad
\Lambda\ge\Lambda_0 .
\]
Enlarging $C_\theta$ if necessary, we may assume that
$C e^{C_\theta'}\le C_\theta$ and that the term $\norm z_{L^2(\Omega)} e^{-\tau_0\Lambda}$
is absorbed into $C_\theta \norm z_{L^2(\Omega)} e^{-\frac{\theta\tau_0}{2}\Lambda}$. Therefore,
\begin{equation}
\label{eq:interp-reduced}
\norm y_{L^2(\Omega)}
\le C_\theta\sup_{x\in E}|y| e^{\frac{1-\theta}{2}\tau_0\Lambda}
+C_\theta \norm z_{L^2(\Omega)} e^{-\frac{\theta\tau_0}{2}\Lambda}.
\end{equation}

If $\sup_{x\in E}|y|\ge \norm z_{L^2(\Omega)}$, the semigroup
contraction gives
\[
\|y\|_{L^2(\Omega)}
\le \norm z_{L^2(\Omega)}
\le \norm z_{L^2(\Omega)}^{1-\theta}
\left(\sup_{x\in E}|y|\right)^\theta .
\]
If $\sup_{x\in E}|y|=0$, letting $\Lambda\to\infty$ in
\eqref{eq:interp-reduced} gives $y=0$.
Now suppose
\[
0<\sup_{x\in E}|y|<\norm z_{L^2(\Omega)}
\]
and put
\[
X_*:=\frac{2}{\tau_0}
\log\frac{\norm z_{L^2(\Omega)}}{\sup_{x\in E}|y|}.
\]
When $X_*\ge\Lambda_0$, choose $\Lambda=X_*$ in
\eqref{eq:interp-reduced}. Then
\[
\norm y_{L^2(\Omega)}
\le
C_\theta
\sup_{x\in E}|y|\,
e^{\frac{(1-\theta)\tau_0}{2}X_*}
+
C_\theta
\norm z_{L^2(\Omega)}
e^{-\frac{\theta\tau_0}{2}X_*}.
\]
Since
\[
X_*=\frac{2}{\tau_0}
\log\frac{\norm z_{L^2(\Omega)}}{\sup_{x\in E}|y|},
\]
we have
\[
e^{\frac{(1-\theta)\tau_0}{2}X_*}
=
\left(
\frac{\norm z_{L^2(\Omega)}}{\sup_{x\in E}|y|}
\right)^{1-\theta},
\qquad
e^{-\frac{\theta\tau_0}{2}X_*}
=
\left(
\frac{\norm z_{L^2(\Omega)}}{\sup_{x\in E}|y|}
\right)^{-\theta}.
\]
Therefore both terms are bounded by
\[
C_\theta
\norm z_{L^2(\Omega)}^{1-\theta}
\left(\sup_{x\in E}|y|\right)^\theta .
\]
When $X_*<\Lambda_0$, we have
\[
\frac{\sup_{x\in E}|y|}{\norm z_{L^2(\Omega)}}
>
e^{-\frac{\tau_0}{2}\Lambda_0}.
\]
The contraction estimate therefore yields
\[
\|y\|_{L^2(\Omega)}
\le \norm z_{L^2(\Omega)}
\le
e^{\frac{\theta\tau_0}{2}\Lambda_0}
\norm z_{L^2(\Omega)}^{1-\theta}
\left(\sup_{x\in E}|y|\right)^\theta .
\]
These cases prove
\[
\norm y_{L^2(\Omega)}
\le C_\theta
\norm z_{L^2(\Omega)}^{1-\theta}
\left(\sup_{x\in E}|y|\right)^\theta .
\]
This proves \eqref{eq:terminal-interp}.
\end{proof}

\section{Stability results}
\label{sec:multiplier}

In this section, we prove the conditional logarithmic stability
estimate and then show that terminal observations on $E$ may fail
to determine the source uniquely when the forcing remains active
at the terminal time.

\subsection{Logarithmic stability}

We first derive a logarithmic estimate from the full terminal
state, then use Proposition~\ref{prop:terminal-interp} to express
the bound in terms of observations on $E$.

\begin{proof}[Proof of Theorem~\ref{thm:log}]
We first prove the logarithmic stability estimate
\eqref{eq:log-main}. By linearity, let
\[
f:=f_1-f_2,
\qquad
y:=u_1(T)-u_2(T).
\]

Using \eqref{eq:switch-off} and \eqref{eq:Duhamel}, we have 
\begin{equation*}
y=e^{-\tau_0A}z,
\qquad
z:=\int_0^{T_0}\rho(r)e^{-(T_0-r)A}f\,dr.
\end{equation*}
The a priori bound and the contraction property give
\begin{equation*}
\norm{z}_{L^2(\Omega)}
\le \norm{\rho}_{L^1(0,T_0)}\norm f_{L^2(\Omega)}\le CM.
\end{equation*}

If $\sup_{x\in E}|y|=0$, Proposition~\ref{prop:terminal-interp}
gives $y=0$ in $L^2(\Omega)$. By the spectral representation
\eqref{eq:spectral-u} at $t=T$ and the definition of $b_k$
in Assumption~\ref{ass:rho}, we have
\[
y=\sum_{k\ge1}b_k f_k\phi_k
\quad\text{in }L^2(\Omega).
\]
Consequently, $b_k f_k=0$ for every $k$. Since
Assumption~\ref{ass:rho} implies $|b_k|>0$ for every $k$,
we obtain $f=0$, as required.

In the remainder of the proof, assume
$\sup_{x\in E}|y|>0$. Since $y$ is continuous in $\Omega$
by Lemma~\ref{lem:interior-smooth}, the identity
$\|y\|_{L^2(\Omega)}=0$ would imply $y\equiv0$ in $\Omega$,
contradicting $\sup_{x\in E}|y|>0$. Hence
$\|y\|_{L^2(\Omega)}>0$.

For \(\lambda_k\le \Lambda\), the multiplier lower bound and
\(\lambda_k\le \Lambda\) give
\[
|b_k|\ge c_b\lambda_k^{-a}e^{-\tau_0\lambda_k}
\ge c_b\Lambda^{-a}e^{-\tau_0\Lambda}.
\]
Summing over \(\lambda_k\le\Lambda\) yields
\[
\|\Pi_\Lambda f\|_{L^2(\Omega)}^2
\le c_b^{-2}\Lambda^{2a}e^{2\tau_0\Lambda}
\sum_{\lambda_k\le\Lambda}|b_kf_k|^2
\le c_b^{-2}\Lambda^{2a}e^{2\tau_0\Lambda}
\|y\|_{L^2(\Omega)}^2,
\]
Taking square roots gives
\[
\|\Pi_\Lambda f\|_{L^2(\Omega)}
\le c_b^{-1}\Lambda^a e^{\tau_0\Lambda}
\|y\|_{L^2(\Omega)}.
\]

The spectral Sobolev prior gives
\[
\norm{(I-\Pi_\Lambda)f}_{L^2(\Omega)}
\le M\Lambda^{-\sigma}.
\]
Consequently,
\begin{equation}
\label{eq:source-split}
\norm f_{L^2(\Omega)}\le
C\Lambda^{a}e^{\tau_0\Lambda}
\|y\|_{L^2(\Omega)}
+M\Lambda^{-\sigma}.
\end{equation}
Let
\[
C_0:=\exp\bigl(
-\max\{1,2\tau_0\lambda_1\}
\bigr).
\]
If $\|y\|_{L^2(\Omega)}\ge C_0M$, the a priori bound
on $\norm f_{L^2(\Omega)}$ and
$\log(e+M/\|y\|_{L^2(\Omega)})
\le\log(e+C_0^{-1})$ give
\[
\norm f_{L^2(\Omega)}\le CM
\left[
\log\left(e+\frac{M}{\|y\|_{L^2(\Omega)}}\right)
\right]^{-\sigma}.
\]
If $0<\|y\|_{L^2(\Omega)}<C_0M$, put
$L:=\log(M/\|y\|_{L^2(\Omega)})$ and choose
$\Lambda=L/(2\tau_0)$. Then
$L>\max\{1,2\tau_0\lambda_1\}$,
so $\Lambda>\lambda_1$.
Estimate \eqref{eq:source-split} becomes
\[
\norm f_{L^2(\Omega)}
\le CM\left(
L^a e^{-L/2}+L^{-\sigma}
\right)
\le CM L^{-\sigma},
\]
where the last inequality follows from \(\sup_{L\ge1}
L^{a+\sigma}e^{-L/2}<\infty.\)

Since
$L\sim\log(e+M/\|y\|_{L^2(\Omega)})$
for $L\ge1$, both cases yield
\begin{equation}
\label{eq:full-terminal-in-proof}
\norm f_{L^2(\Omega)}\le CM
\left[
\log\left(e+\frac{M}{\|y\|_{L^2(\Omega)}}\right)
\right]^{-\sigma}.
\end{equation}

Next, we apply Proposition~\ref{prop:terminal-interp}
with a fixed $\theta\in(0,1)$ to
$y=e^{-\tau_0A}z$. Using the bound
on $z$, 
\[
\|y\|_{L^2(\Omega)}
\le CM^{1-\theta}\left(\sup_{x\in E}|y|\right)^\theta,
\qquad
\frac{M}{\|y\|_{L^2(\Omega)}}
\ge C^{-1}
\left(\frac{M}{\sup_{x\in E}|y|}\right)^\theta.
\]
There exists
$c_\theta>0$ such that
\[
\log\left(e+\frac{M}{\|y\|_{L^2(\Omega)}}\right)
\ge\log\left(e+C^{-1}\left(\frac{M}{\sup_{x\in E}|y|}\right)^\theta\right)
\ge c_\theta\log\left(e+\frac{M}{\sup_{x\in E}|y|}\right).
\]
Indeed, the logarithms are comparable
for $0\le M/\sup_{x\in E}|y(x)|\le1$, while their ratio
tends to $\theta$ as $M/\sup_{x\in E}|y(x)|\to\infty$
and stays positive on every bounded interval.
Combining this inequality with
\eqref{eq:full-terminal-in-proof}
proves \eqref{eq:log-main}.
\end{proof}

\subsection{Noninjectivity for a source active at the terminal time}
Here we drop the condition \eqref{eq:switch-off}; the forward solution remains defined by \eqref{eq:Duhamel}. Without a source-free waiting interval, the terminal multiplier need not provide exponential spatial smoothing.

\begin{proposition}
Let $E\Subset\Omega$ be compact. Suppose $\rho\in W^{1,\infty}(0,T)$ and its Lipschitz representative satisfies $\rho(T)\ne0$. There exists a nonzero source
\(f\in\bigcap_{m\ge0}D(A^m)\), whose solution of \eqref{eq:model} satisfies $u(\cdot,T)|_E=0$. Consequently, for every finite $\sigma\ge0$ and every $M>0$, the terminal observation on $E$ is noninjective even among sources satisfying $\|A^\sigma f\|_{L^2(\Omega)}\le M$.
\end{proposition}

\begin{proof}
Denote the active-terminal multipliers by
\[
b_k^{\prime}:=\int_0^T\rho(t)e^{-\lambda_k(T-t)}\,dt.
\]
Integration by parts, justified by $\rho\in W^{1,\infty}(0,T)$, gives
\[
b_k^{\prime}
=\frac{\rho(T)}{\lambda_k}
-\frac{\rho(0)e^{-T\lambda_k}}{\lambda_k}
-\frac1{\lambda_k}\int_0^T\rho'(t)e^{-\lambda_k(T-t)}\,dt
=\frac{\rho(T)}{\lambda_k}+O(\lambda_k^{-2}).
\]
Thus the set $\mathcal Z:=\{k:b_k^{\prime}=0\}$ is finite, and there is a constant $C$ such that
\[
|b_k^{\prime}|^{-1}\le C\lambda_k
\qquad(k\notin\mathcal Z).
\]
Choose a nonempty ball $B\Subset\Omega\setminus E$. Since $C_c^\infty(B)$ is infinite-dimensional, there exists $0\ne\psi\in C_c^\infty(B)$ such that $(\psi,\phi_k)=0$ for every $k\in\mathcal Z$. Let $\psi_k=(\psi,\phi_k)$ and set
\[
f_k:=
\begin{cases}
\psi_k/b_k^{\prime},&k\notin\mathcal Z,\\
0,&k\in\mathcal Z.
\end{cases}
\]
Because $\psi\in C_c^\infty(\Omega)\subset\bigcap_{m\ge0}D(A^m)$, for each $m\ge0$ we have
\[
\sum_{k\ge1}\lambda_k^{2m}|f_k|^2
\le C^2\sum_{k\ge1}\lambda_k^{2(m+1)}|\psi_k|^2<\infty.
\]
Hence $f=\sum_{k\ge1}f_k\phi_k$ belongs to $\bigcap_{m\ge0}D(A^m)$ and is nonzero. Its terminal state has coefficients $b_k^{\prime}f_k=\psi_k$ for every $k$, so $u(T)=\psi$. Since $\psi$ vanishes on $E$, its terminal observation is zero. Multiplying $f$ by a sufficiently small positive constant places it in any prescribed ball $\|A^\sigma f\|_{L^2}\le M$ without changing the vanishing observation.
\end{proof}

\section{Scattered-data reconstruction and convergence analysis}
\label{sec:numerical}

In this section, we establish sampling estimates and analyze the
convergence of Tikhonov reconstructions from scattered observations
on $E$.
Throughout this section, $d\ge1$, and \eqref{eq:haus},
\eqref{eq:switch-off}, and Assumption~\ref{ass:rho} are in force.
The symbol $\beta$ denotes the Hausdorff content exponent, while
$\sigma>0$ denotes the source regularity and penalty exponent.

\subsection{Scattered-data Tikhonov reconstruction}

For $R\ge\lambda_1$, define
\[
V_R=\operatorname{span}\{\phi_k:\lambda_k\le R\},
\qquad N_R=\dim V_R.
\]
The Weyl counting estimate gives $N_R\le C(1+R)^{d/2}$.
Let $X_n=\{x_{n,1},\ldots,x_{n,n}\}\subset E$ consist of pairwise
distinct deterministic sites, and set
\[
S_nf=\big(u(x_{n,j},T)\big)_{j=1}^n,
\qquad
\|w\|_n^2=\frac1n\sum_{j=1}^n|w_j|^2,
\qquad
(w,z)_n=\frac1n\sum_{j=1}^nw_jz_j.
\]
For a continuous function $v$ on $E$, the notation $\|v\|_n$ refers
to the same norm of its sampled values. We use real-valued sources
and data.
The observed data are
\begin{equation}\label{qu-eq:data}
m_{n,j}=u(x_{n,j},T)+\eta_{n,j},
\qquad j=1,\ldots,n.
\end{equation}
The positive waiting time makes point evaluation bounded on terminal
states, so $S_n$ is well defined on $L^2(\Omega)$.

For $\alpha>0$, the reconstruction is the following Tikhonov minimizer
\begin{equation}\label{eq:discrete-tikhonov}
f_{n,R,\alpha}
=\mathop{\rm argmin}_{g\in V_R}
\left\{\|S_ng-m_n\|_n^2
+\alpha\|A^\sigma g\|_{L^2(\Omega)}^2\right\}.
\end{equation}
The functional is strictly convex and coercive on $V_R$, and hence
has a unique minimizer.

To express the minimization problem in matrix form, write
$g=\sum_{\lambda_k\le R}c_k\phi_k$ and define
\[
(H_{n,R})_{jk}=b_k\phi_k(x_{n,j}).
\]
The coefficient vector of the minimizer is determined by
\begin{equation}\label{eq:sampled-normal}
\begin{aligned}
c_\alpha
&=
\left[
\frac1n H_{n,R}^{\mathsf T}H_{n,R}
+\alpha\operatorname{diag}(\lambda_k^{2\sigma})
\right]^{-1}
\frac1n H_{n,R}^{\mathsf T}m_n,\quad
f_{n,R,\alpha}=\sum_{\lambda_k\le R}(c_\alpha)_k\phi_k,
\end{aligned}
\end{equation}
where the diagonal matrix is indexed by the modes
$\lambda_k\le R$.

Define
\begin{equation}\label{eq:clz-exponent}
q_\sigma=\frac{8(1+\sigma)}{4(1+\sigma)+d},\qquad
p_0=\frac{q_\sigma}{2}
=\frac{4(1+\sigma)}{4(1+\sigma)+d}\in(0,1).
\end{equation}
The parameter update is
\begin{equation}\label{eq:clz-fixed-point}
\alpha^{(j+1)}=\Phi_{n,R}(\alpha^{(j)}),\qquad
\Phi_{n,R}(\alpha)=
\left(
n^{-1/2}
\frac{\|H_{n,R}c_\alpha-m_n\|_n}
{{\|A^\sigma f_{n,R,\alpha}\|_{L^2(\Omega)}}}
\right)^{q_\sigma}.
\end{equation}
Fix a constant $c_0>0$, independent of $n$ and the
observed data. The initial value used in the iteration is
$\alpha^{(0)}=c_0n^{-p_0}$. At each step, \eqref{eq:sampled-normal} determines the regularized
solution, and \eqref{eq:clz-fixed-point} updates the parameter from its residual and penalty norm. 

For data satisfying $H_{n,R}^{\mathsf T}m_n\ne0$, both
$\|H_{n,R}c_\alpha-m_n\|_n$ and
$\|A^\sigma f_{n,R,\alpha}\|_{L^2(\Omega)}$ are positive for
$\alpha>0$, and $\Phi_{n,R}$ is continuous. For a prescribed tolerance $\mathrm{tol}_\alpha>0$,
the iteration is stopped once
\begin{equation*}
|\alpha^{(j+1)}-\alpha^{(j)}|\le\mathrm{tol}_\alpha.
\end{equation*}
If the stopping criterion is not met within the prescribed maximum number of iterations, the last iterate is returned. The output is the selected parameter $\alpha^{(j+1)}$ and the corresponding reconstruction $f_{n,R,\alpha^{(j+1)}}$, computed from \eqref{eq:sampled-normal}.

\subsection{Quasi-uniform sampling}

For the growing-sample analysis, the sites are assumed to be
quasi-uniform relative to the fixed compact set $E\Subset\Omega$. For $n\ge2$, define
\[
h_n=\sup_{x\in E}\min_{1\le j\le n}|x-x_{n,j}|,
\qquad
q_n=\min_{i\ne j}|x_{n,i}-x_{n,j}|,
\]
and assume that $h_n\le Bq_n$, where $B$ is independent of $n$. Recall that
\begin{equation}\label{qu-eq:content}
\mathcal H_\infty^\beta(E)>0,
\qquad d-1<\beta\le d.
\end{equation}

\begin{lemma}[Quasi-uniform sampling under positive Hausdorff content]
\label{qu-lem:sampling}
For each fixed $L>0$, there exist $C_L>0$ and $n_L\ge1$
such that
\begin{equation}\label{qu-eq:sampling}
\|v\|_{L^2(\Omega)}
\le C_L e^{C_L\sqrt R}
(nh_n^\beta)^{1/2}\|v\|_n
\end{equation}
for every $v\in V_R$, $1\le R\le L\log n$, and $n\ge n_L$.
The constants may depend on $E,\beta,\Omega,B,L$, but not on
$R,n,v$.
\end{lemma}

\begin{proof}
To make the proof easier to follow, we divide the proof into three steps.

\medskip
\noindent
\emph{Step 1. Packing estimate and a Frostman measure.}
The balls $B(x_{n,j},q_n/3)$ are pairwise disjoint and contained
in a fixed bounded enlargement of $E$. Comparing their total
$d$-dimensional volume with that of this enlargement gives
\(nq_n^d\le C.\)
Consequently, the quasi-uniformity assumption implies
\begin{equation}\label{qu-eq:fill-bound}
h_n\le Bq_n\le Cn^{-1/d}.
\end{equation}

By Frostman's lemma
\cite[Theorem~8.8]{Mattila95}, the positive
Hausdorff-content condition \eqref{qu-eq:content}
yields a probability measure $\mu$ supported on $E$
and a constant $C_F>0$ such that
\begin{equation}\label{qu-eq:frostman}
\mu(B(x,r))\le C_Fr^\beta,
\qquad x\in\mathbb R^d,\quad r>0.
\end{equation}
We first prove the measure norming inequality
\begin{equation}\label{qu-eq:munorm}
\|v\|_{L^2(\Omega)}
\le Ce^{C\sqrt R}\|v\|_{L^2(E,\mu)},
\qquad v\in V_R,\quad R\ge1.
\end{equation}

\medskip
\noindent
\emph{Step 2. Uniform spectral inequality for subsets.}
Choose finitely many fixed closed balls
$B_j=\overline B(x_j,r_j)$, $1\le j\le J$, such that
\[
E\subset\bigcup_{j=1}^J B_j,
\qquad 8B_j\Subset\Omega.
\]
Here $8B_j$ denotes the concentric ball with eight times the
radius of $B_j$. Such a cover exists because $E\Subset\Omega$.

We establish the following auxiliary estimate. For every fixed
$m_0>0$, there is a constant $C>0$ such that
\begin{equation}\label{qu-eq:uniform-subset}
\|v\|_{L^2(\Omega)}
\le Ce^{C\sqrt R}\sup_{x\in F}|v(x)|
\end{equation}
for every $R\ge1$, every $v\in V_R$, and every compact set
$F\subset B_j$ satisfying
\[
\mathcal H_\infty^\beta(F)\ge m_0,
\qquad 1\le j\le J.
\]
The constant in \eqref{qu-eq:uniform-subset} is independent of
the particular set $F$.

To prove this estimate, fix one of the balls $B_j$ and a compact
set $F\subset B_j$ with
$\mathcal H_\infty^\beta(F)\ge m_0$.
By the spectral inequality on the fixed ball $2B_j$
(see \cite[Theorem~2.1]{Green25}), there exists $C_j>0$
such that
\begin{equation}\label{qu-eq:fixed-ball}
\|v\|_{L^2(\Omega)}
\le C_j e^{C_j\sqrt R}
\sup_{x\in 2B_j}|v(x)|,
\qquad v\in V_R.
\end{equation}
Here $C_j$ depends only on the fixed geometry, not on $F$.

For $v\in V_R$, write
\(v=\sum_{\lambda_k\le R}v_k\phi_k\)
and introduce the finite harmonic extension
\begin{equation*}
U_v(x,s)
=
\sum_{\lambda_k\le R}
v_k\frac{\sinh(\sqrt{\lambda_k}s)}{\sqrt{\lambda_k}}
\phi_k(x).
\end{equation*}
Since $-\Delta\phi_k=\lambda_k\phi_k$, direct differentiation
shows that
\[
(\Delta_x+\partial_s^2)U_v=0
\qquad\text{in }\Omega\times\mathbb R.
\]
Moreover,
\[
U_v(x,0)=0,\qquad
\nabla_xU_v(x,0)=0,\qquad
\partial_sU_v(x,0)=v(x),
\]
and hence
\begin{equation}\label{qu-eq:lift-trace}
|\nabla_{x,s}U_v(x,0)|=|v(x)|.
\end{equation}

For each fixed $a>0$, orthogonality of the eigenfunctions gives
\begin{align*}
\|U_v\|_{L^2(\Omega\times(-a,a))}^2
&=
\sum_{\lambda_k\le R}|v_k|^2
\int_{-a}^a
\frac{\sinh^2(\sqrt{\lambda_k}s)}{\lambda_k}\,ds
\le C_a e^{2a\sqrt R}
\|v\|_{L^2(\Omega)}^2.
\end{align*}
Here we used $\lambda_k\ge\lambda_1>0$ and
$\lambda_k\le R$ for the modes in the sum.

After translation and rescaling, we apply the gradient estimate
for propagation of smallness \cite[Lemma~2.3]{Green25}.
The ambient dimension is $d+1$, and we choose
\[
\delta_*:=\beta-(d-1)>0,
\]
so that $(d+1)-2+\delta_*=\beta$. For each fixed $j$, we
apply the lemma with the compact sets
\[
K=2B_j\times\{0\},\qquad F\times\{0\},
\]
and concentric balls centered at $(x_j,0)$ with radii
$4r_j$ and $8r_j$, respectively. Since $F\subset B_j$
and $8B_j\Subset\Omega$, the required geometric inclusions
hold, and $U_v$ is harmonic on the larger ball. Moreover,
\[
\mathcal H_\infty^\beta(F\times\{0\})
=\mathcal H_\infty^\beta(F)\ge m_0.
\]
Thus the geometric parameters and content lower bound are
uniform in $F$. Since there are only finitely many balls
$B_j$, the resulting constants $C_j>0$ and
$\kappa_j\in(0,1)$ can be chosen uniformly over all admissible sets $F$.

The propagation estimate consequently yields
\begin{equation}\label{qu-eq:gradient-propagation}
\sup_{x\in 2B_j}|v(x)|
\le C_j
\left(\sup_{x\in F}|v(x)|\right)^{\kappa_j}
\|\nabla U_v\|_{L^\infty(\mathbb B_{4r_j}^{d+1})}^{1-\kappa_j}.
\end{equation}

The interior estimates for harmonic functions, together with
$6B_j\Subset\Omega$, imply
\begin{align}
\|\nabla U_v\|_{L^\infty(\mathbb B_{4r_j}^{d+1})}
&\le C_j
\|U_v\|_{L^2(\mathbb B_{6r_j}^{d+1})}
\nonumber\\
&\le C_j
\|U_v\|_{L^2(\Omega\times(-6r_j,6r_j))}
\nonumber\\
&\le C_j e^{C_j\sqrt R}\|v\|_{L^2(\Omega)}.
\label{qu-eq:lift-gradient-bound}
\end{align}
The fixed geometry of the balls is essential in these estimates.

Applying \eqref{qu-eq:gradient-propagation} and
\eqref{qu-eq:lift-trace}--\eqref{qu-eq:lift-gradient-bound},
we obtain
\[
\sup_{x\in 2B_j}|v(x)|
\le C_j e^{C_j\sqrt R}
\left(\sup_{x\in F}|v(x)|\right)^{\kappa_j}
\|v\|_{L^2(\Omega)}^{1-\kappa_j}.
\]
Combining this with \eqref{qu-eq:fixed-ball} gives
\[
\|v\|_{L^2(\Omega)}
\le C_j e^{C_j\sqrt R}
\left(\sup_{x\in F}|v(x)|\right)^{\kappa_j}
\|v\|_{L^2(\Omega)}^{1-\kappa_j}.
\]
If $v\ne0$, division by
$\|v\|_{L^2(\Omega)}^{1-\kappa_j}$ and raising both sides
to the power $1/\kappa_j$ yield
\[
\|v\|_{L^2(\Omega)}
\le C_j e^{C_j\sqrt R}\sup_{x\in F}|v(x)|,
\]
after changing the constant. The case $v=0$ is immediate.

Although the ball index and the exponent $\kappa_j$ may vary,
there are only finitely many balls. Taking the maximum of the
resulting constants proves \eqref{qu-eq:uniform-subset} uniformly
over all such $F$ and all $j$.

We now apply \eqref{qu-eq:uniform-subset} to prove
\eqref{qu-eq:munorm}. Let $v\in V_R$ be nonzero and put
\[
s_v:=\|v\|_{L^2(E,\mu)},\qquad
G_v:=\{x\in E:|v(x)|\le2s_v\}.
\]
Since $v$ is continuous on the compact set $E$, the set $G_v$
is compact. If $s_v>0$, Chebyshev's inequality gives
\begin{align*}
\mu(E\setminus G_v)
&=\mu\bigl(\{x\in E:|v(x)|>2s_v\}\bigr)
\le \frac{1}{4s_v^2}\int_E|v|^2\,d\mu
=\frac14.
\end{align*}
Thus $\mu(G_v)\ge3/4$. If $s_v=0$, then $v=0$
$\mu$-almost everywhere and $\mu(G_v)=1$.
In either case, \(\mu(G_v)\ge\frac34.\)

Because the balls $B_1,\ldots,B_J$ cover $E$, there exists
$j=j(v)$ such that
\[
\mu(G_v\cap B_j)\ge\frac{3}{4J}.
\]
Set $F_v:=G_v\cap B_j$. Then $F_v$ is compact. For any
countable covering $\{B(x_i,r_i)\}_i$ of $F_v$, we have,
by \eqref{qu-eq:frostman},
\[
\mu(F_v)
\le \sum_i\mu(B(x_i,r_i))
\le C_F\sum_i r_i^\beta.
\]
Taking the infimum over all such coverings, we obtain
\[
\mathcal H_\infty^\beta(F_v)
\ge \frac{\mu(F_v)}{C_F}
\ge \frac{3}{4JC_F}>0.
\]
Since this lower bound is independent of $v$,
\eqref{qu-eq:uniform-subset} applies to $F_v$ with a constant
independent of $v$. Hence,
\[
\|v\|_{L^2(\Omega)}
\le Ce^{C\sqrt R}\sup_{x\in F_v}|v(x)|
\le 2Ce^{C\sqrt R}\|v\|_{L^2(E,\mu)}.
\]
Absorbing the factor $2$ into $C$ proves
\eqref{qu-eq:munorm}. If $s_v=0$, the same argument forces
$v=0$ in $L^2(\Omega)$, so the inequality also holds in that
case. The inequality for $v=0$ is trivial.

\medskip
\noindent
\emph{Step 3. Discretization and absorption of the sampling error.}
Partition $E$ into Borel nearest-neighbor cells
$D_{n,j}$, resolving ties by the smallest index, and set
\[
w_{n,j}:=\mu(D_{n,j}),\qquad 1\le j\le n.
\]
Every cell satisfies
\(D_{n,j}\subset\overline B(x_{n,j},h_n).\)
Therefore, by \eqref{qu-eq:frostman},
\begin{equation}\label{qu-eq:cell-weight}
w_{n,j}\le C_Fh_n^\beta.
\end{equation}

Choose a fixed compact neighborhood $K\Subset\Omega$ of $E$.
Interior spectral elliptic estimates imply that, for some integer
$m\ge0$,
\begin{equation*}
\|\nabla v\|_{L^\infty(K)}
\le C(1+R)^m\|v\|_{L^2(\Omega)},
\qquad v\in V_R.
\end{equation*}
Since $h_n\to0$, for all sufficiently large $n$ the line segment
joining any $x\in D_{n,j}$ to $x_{n,j}$ lies in $K$. The
mean-value estimate then gives
\[
|v(x)-v(x_{n,j})|
\le h_n\|\nabla v\|_{L^\infty(K)},
\qquad x\in D_{n,j}.
\]
Using the triangle inequality in $L^2(E,\mu)$ and
$\mu(E)=1$, we obtain
\begin{align*}
\|v\|_{L^2(E,\mu)}
&\le
\left(\sum_{j=1}^n w_{n,j}|v(x_{n,j})|^2\right)^{1/2}
+h_n\|\nabla v\|_{L^\infty(K)}\\
&\le C(nh_n^\beta)^{1/2}\|v\|_n
+C h_n(1+R)^m\|v\|_{L^2(\Omega)}.
\end{align*}
Here the first term was estimated using \eqref{qu-eq:cell-weight}.

Combining the last estimate with \eqref{qu-eq:munorm} gives
\begin{align}
\|v\|_{L^2(\Omega)}
&\le Ce^{C\sqrt R}(nh_n^\beta)^{1/2}\|v\|_n
+Ch_n(1+R)^m e^{C\sqrt R}
\|v\|_{L^2(\Omega)}.
\label{qu-eq:before-absorption}
\end{align}
By \eqref{qu-eq:fill-bound}, the coefficient of the last term
is bounded uniformly for $1\le R\le L\log n$ by
\[
Cn^{-1/d}(1+L\log n)^m e^{C\sqrt{L\log n}}
\longrightarrow0
\qquad(n\to\infty).
\]
Indeed, the logarithm of this expression is bounded above by
\[
-\frac1d\log n+C\sqrt{L\log n}
+m\log(1+L\log n)+C,
\]
which tends to $-\infty$. Hence the convergence is uniform
over the entire spectral range $1\le R\le L\log n$.

Choose $n_L$ sufficiently large that this coefficient is at
most $1/2$. Absorbing the last term of
\eqref{qu-eq:before-absorption} into the left-hand side and
enlarging the constant yields
\[
\|v\|_{L^2(\Omega)}
\le C_L e^{C_L\sqrt R}
(nh_n^\beta)^{1/2}\|v\|_n.
\]
This proves \eqref{qu-eq:sampling}.
\end{proof}

\subsection{Stochastic convergence analysis}

In this subsection, constants may also depend on the fixed source,
the a priori bound $M$, the noise distribution, and $c_0$, but are
independent of $n$ and the noise realization.

Let $f^\dagger\in D(A^\sigma)\setminus\{0\}$ satisfy
$\|A^\sigma f^\dagger\|_{L^2(\Omega)}\le M$. In the data model
\eqref{qu-eq:data}, assume that all noise variables have the same
distribution, independent of $n$, with mean zero and variance
$\nu^2\in(0,\infty)$. They are independent within each row.

With $p_0$ as in \eqref{eq:clz-exponent}, choose
$\kappa>p_0/(2\tau_0)$ and set
\[
K_n=\kappa\log n,\qquad f_{n,\alpha}=f_{n,K_n,\alpha}.
\]
For all sufficiently large $n$, $K_n\ge\lambda_1$, so $f_{n,\alpha}$
is defined by \eqref{eq:discrete-tikhonov}.

\begin{theorem}[Convergence analysis]
Under \eqref{qu-eq:content}, quasi-uniformity, and
\eqref{eq:multiplier-two-sided}, assume the noise model
\eqref{qu-eq:data} and $p_0>1-\beta/d$.
For any fixed $c_0>0$, independent of $n$ and the observed data,
there exists a constant $C>0$, independent of $n$ and the noise
realization, such that, with probability tending to one, the
iteration \eqref{eq:clz-fixed-point} with $R=K_n$ and
$\alpha^{(0)}=c_0n^{-p_0}$ converges monotonically to a positive
fixed point $\alpha_*$, and the corresponding reconstruction satisfies
\[
\big\|f_{n,\alpha_*}-f^\dagger\big\|_{L^2(\Omega)}
\le C(\log n)^{-\sigma}.
\]
\end{theorem}

\begin{proof}
The proof consists of three steps. We first derive uniform
prediction and regularity estimates for the finite-dimensional
Tikhonov reconstructions. We then use the sampling inequality
to convert the prediction error into a source reconstruction
error. Finally, we show that the fixed-point parameter update
preserves a suitable interval and converges to a positive fixed
point within that interval.

\medskip
\noindent
\emph{Step 1: terminal remainder and stochastic prediction estimates.}
Let $f_n^0=\Pi_{K_n}f^\dagger$, and let
\[
\zeta_n=S_n(I-\Pi_{K_n})f^\dagger.
\]
We first estimate the contribution of the source components
above the spectral cutoff.
To this end, write
\[
u_v(\cdot,T)=e^{-\tau_0 A}F(A)v,\qquad
F(A)=\int_0^{T_0}\rho(t)e^{-(T_0-t)A}\,dt,
\qquad |F(\lambda)|\le\|\rho\|_{L^\infty}/\lambda.
\]
For $v\in L^2(\Omega)$ and $R\ge R_{\rm tail}:=\max\{1,2/\tau_0\}$,
put $s=1/R<\tau_0$. Since the spectral factors commute,
\[
(I-\Pi_R)u_v(\cdot,T)
=e^{-sA}e^{-(\tau_0-s)A}F(A)(I-\Pi_R)v.
\]
The heat-semigroup estimate
$\|e^{-sA}\|_{L^2(\Omega)\to L^\infty(\Omega)}\le Cs^{-d/4}$
and the spectral theorem give
\begin{align*}
\sup_{x\in E}\big|(I-\Pi_R)u_v(x,T)\big|
&\le C R^{d/4}
\sup_{\lambda>R}\bigl(|F(\lambda)|e^{-(\tau_0-1/R)\lambda}\bigr)
\|(I-\Pi_R)v\|_2\\
&\le C R^{d/4-1}e^{-(\tau_0-1/R)R}\|(I-\Pi_R)v\|_2.
\end{align*}
Since $e^{-(\tau_0-1/R)R}=e\,e^{-\tau_0R}$, taking
$p=\max\{d/4-1,0\}$ and absorbing $e$ into the constant yields
\begin{equation}\label{qu-eq:tail}
\sup_{x\in E}\big|(I-\Pi_R)u_v(x,T)\big|
\le C(1+R)^p e^{-\tau_0R}\|(I-\Pi_R)v\|_2.
\end{equation}
Applying \eqref{qu-eq:tail} with $R=K_n$ and $v=f^\dagger$, and using
$\|(I-\Pi_{K_n})f^\dagger\|_2\le M K_n^{-\sigma}$, we obtain
\[
\|\zeta_n\|_n
\le C M(1+K_n)^p e^{-\tau_0K_n}K_n^{-\sigma}.
\]
Thus $\|\zeta_n\|_n^2=o(n^{-p_0})$, since $2\tau_0\kappa>p_0$.

Let $P_n$ denote the Euclidean orthogonal projection onto
$S_n(V_{K_n})$. Since $(\cdot,\cdot)_n$ is the Euclidean
inner product multiplied by $1/n$, $P_n$ is also the orthogonal
projection with respect to $(\cdot,\cdot)_n$.
For $g\in V_{K_n}$, the terminal state
$u_g(\cdot,T)$ also belongs to $V_{K_n}$. Since the terminal
multipliers are nonzero, Lemma~\ref{qu-lem:sampling} implies
that $S_n|_{V_{K_n}}$ is injective for large $n$. Thus
$\dim S_n(V_{K_n})=N_{K_n}=O((\log n)^{d/2})$.

Set $g_{n,\alpha}=S_n(f_{n,\alpha}-f_n^0)$. Minimality against
$f_n^0$ and $S_n f_n^0-m_n=-(\zeta_n+\eta_n)$ give
\[
\norm{g_{n,\alpha}}_n^2+\alpha\norm{A^\sigma f_{n,\alpha}}_2^2
\le 2( g_{n,\alpha},\zeta_n+\eta_n)_n
+\alpha M^2.
\]
Because $g_{n,\alpha}$ lies in ${S_n(V_{K_n})}$,
we have $( g_{n,\alpha},\eta_n)_n
=( g_{n,\alpha},P_n\eta_n)_n$. Consequently,
\[
2( g_{n,\alpha},\zeta_n+\eta_n)_n
\le \frac12\|g_{n,\alpha}\|_n^2
+4\bigl(\|\zeta_n\|_n^2+\|P_n\eta_n\|_n^2\bigr).
\]
Substitution yields, simultaneously for every $\alpha>0$,
\begin{align}
\norm{g_{n,\alpha}}_n^2
&\le C\Big(\alpha M^2+{\|\zeta_n\|_n^2}
+\norm{P_n\eta_n}_n^2\Big),
\label{qu-eq:prediction}\\
\norm{A^\sigma f_{n,\alpha}}_2^2
&\le M^2+\frac C\alpha
\Big({\|\zeta_n\|_n^2}+\norm{P_n\eta_n}_n^2\Big).
\label{qu-eq:normbound}
\end{align}

Independence and equal noise variance imply the exact identity
\[
\mathbb E\norm{P_n\eta_n}_n^2
=\frac{\nu^2}{n}\operatorname{tr}(P_n)
=\nu^2\frac{N_{K_n}}n
=O((\log n)^{d/2}/n).
\]
Since $p_0<1$, Markov's inequality gives
$\|P_n\eta_n\|_n^2=o_{\mathbb P}(n^{-p_0})$.
Together with the bound on $\zeta_n$, equations
\eqref{qu-eq:prediction}--\eqref{qu-eq:normbound} imply that,
uniformly for $\alpha$ in any fixed interval
$[c n^{-p_0},C_0n^{-p_0}]$, where $0<c<C_0<\infty$,
\begin{equation}\label{qu-eq:prediction-rate}
\norm{{S_n}(f_{n,\alpha}-f^\dagger)}_n
=O_{\mathbb P}(n^{-p_0/2}),\quad
\norm{A^\sigma f_{n,\alpha}}_2\le M+o_{\mathbb P}(1).
\end{equation}
The first estimate uses
$S_n(f_{n,\alpha}-f^\dagger)=g_{n,\alpha}-\zeta_n$.

\medskip
\noindent
\emph{Step 2: transfer of the prediction error to the source error.}
Fix arbitrary constants $0<c_-<c_+<\infty$, and consider the
parameter interval
\[
I_n=[c_-n^{-p_0},c_+n^{-p_0}].
\]
We establish a source error estimate uniformly over this interval.
Let $h=f_{n,\alpha}-f^\dagger$ and write
$\delta_n=\norm{{S_n}h}_n$.
On events of probability tending to one,
$\norm{A^\sigma h}_2\le 2M+1$ uniformly in the parameter interval.
Choose
\[
0<u<\frac{p_0-(1-\beta/d)}{2\tau_0},\qquad
v>u+\frac{1-\beta/d}{2\tau_0},
\qquad
U_n=u\log n,\quad R_n=v\log n.
\]
For $n$ large, $1\ll U_n\le R_n$.
Lemma~\ref{qu-lem:sampling} applies at $R_n$ with any fixed
$L\ge v$. The multiplier lower bound
\eqref{eq:multiplier-two-sided}, the spectral inequality
\eqref{qu-eq:sampling} applied to
$\Pi_{R_n}{u_h(\cdot,T)}$, and the remainder estimate
\eqref{qu-eq:tail} yield
\begin{equation}\label{qu-eq:source-transfer}
\begin{aligned}
\norm h_2
&\le \norm{\Pi_{U_n}h}_2+\norm{(I-\Pi_{U_n})h}_2\\
&\le C(1+U_n)^ae^{\tau_0 U_n}
\norm{\Pi_{U_n}{u_h(\cdot,T)}}_2
+(2M+1)U_n^{-\sigma}\\
&\le C(1+U_n)^a e^{\tau_0 U_n+C\sqrt{R_n}}
(nh_n^\beta)^{1/2}
\left[\delta_n+C(1+R_n)^p e^{-\tau_0 R_n}\norm h_2\right]
+(2M+1)U_n^{-\sigma}.
\end{aligned}
\end{equation}
The $L^2$-norm of $h$ in the square brackets is bounded by
$\lambda_1^{-\sigma}(2M+1)$ on the same events.
The first square-bracket term is $O_{\mathbb P}(n^{-p_0/2})$
by \eqref{qu-eq:prediction-rate}; the second is deterministic up
to an $O_{\mathbb P}(1)$ factor. The packing bound
$h_n\le Cn^{-1/d}$ gives
\[
(nh_n^\beta)^{1/2}
\le Cn^{(1-\beta/d)/2}.
\]
Using \(p_0>1-\frac{\beta}{d}\), the choices of $u,v$ imply
\[
\frac{1-\beta/d}{2}+\tau_0 u-\frac{p_0}{2}<0,\qquad
\frac{1-\beta/d}{2}+\tau_0 u-\tau_0 v<0.
\]
Since $e^{C\sqrt{v\log n}}=n^{o(1)}$, both exponentially small
contributions are $o_{\mathbb P}((\log n)^{-\sigma})$.
Therefore,
\begin{equation}\label{qu-eq:uniform}
\sup_{\alpha\in I_n}
\Big\|f_{n,\alpha}-f^\dagger\Big\|_{L^2(\Omega)}
=O_{\mathbb P}((\log n)^{-\sigma}).
\end{equation}

\medskip
\noindent
\emph{Step 3: construction of an invariant interval and convergence
of the parameter iteration.}
We now use the uniform estimates from the first two steps to
construct an interval that contains the initial parameter and is
invariant under the fixed-point map. For any fixed interval
$I_n=[c_-n^{-p_0},c_+n^{-p_0}]$,
\eqref{qu-eq:uniform}, together with the assumption
$f^\dagger\ne0$ and the bound \eqref{qu-eq:normbound}, implies that,
with probability tending to one,
\[
\tfrac12\lambda_1^\sigma\|f^\dagger\|_2\le
\|A^\sigma f_{n,\alpha}\|_2
\le M+1,
\qquad \alpha\in I_n.
\]
Indeed, the lower bound follows from
$\|A^\sigma f_{n,\alpha}\|_2\ge\lambda_1^\sigma
\|f_{n,\alpha}\|_2$ and uniform source convergence; the upper bound
follows from \eqref{qu-eq:normbound}.
Moreover,
\[
\sup_{\alpha\in I_n}
\left|\|S_nf_{n,\alpha}-m_n\|_n-\|\eta_n\|_n\right|
\le\sup_{\alpha\in I_n}\|S_n(f_{n,\alpha}-f^\dagger)\|_n
=o_{\mathbb P}(1).
\]
The weak law of large numbers gives
$\|\eta_n\|_n\to\nu>0$ in probability. Hence, with probability
tending to one, the residual lies between $\nu/2$ and $2\nu$,
uniformly on every fixed interval of this form. It follows that
\[
a_0n^{-p_0}
\le\Phi_{n,K_n}(\alpha)
\le b_0n^{-p_0},\quad
a_0:=\left(\frac{\nu}{2M+2}\right)^{q_\sigma},
\quad
b_0:=\left(\frac{4\nu}{\lambda_1^\sigma\|f^\dagger\|_2}\right)^{q_\sigma}.
\]

Choose
\[
c_*:=\min\{1/2,c_0/2,a_0/2\},\qquad
C_*:=\max\{2,2c_0,2b_0\},\qquad
I_n^*:=[c_*n^{-p_0},C_*n^{-p_0}].
\]
Since $c_*$ and $C_*$ are fixed positive constants independent of
$n$ and the observed data, the estimates established above apply
to $I_n^*$ as well. Consequently, on events whose probabilities
tend to one,
\[
\alpha^{(0)}=c_0n^{-p_0}\in I_n^*,\qquad
\Phi_{n,K_n}(I_n^*)
\subset[a_0n^{-p_0},b_0n^{-p_0}]\subset I_n^*.
\]
Moreover, the denominator in
$\Phi_{n,K_n}$ is bounded below by
$\tfrac12\lambda_1^\sigma\|f^\dagger\|_2$ throughout $I_n^*$,
so $\Phi_{n,K_n}$ is continuous there.

For $0<\alpha_1<\alpha_2$, comparing the two Tikhonov minimizers
in their respective functionals gives
\[
(\alpha_2-\alpha_1)
\bigl(\|A^\sigma f_{n,\alpha_2}\|_2^2
-\|A^\sigma f_{n,\alpha_1}\|_2^2\bigr)\le0,
\]
and
\begin{align*}
\|S_nf_{n,\alpha_1}-m_n\|_n^2
-\|S_nf_{n,\alpha_2}-m_n\|_n^2
\le\alpha_1
\bigl(\|A^\sigma f_{n,\alpha_2}\|_2^2
-\|A^\sigma f_{n,\alpha_1}\|_2^2\bigr)\le0.
\end{align*}
Thus the penalty norm is nonincreasing and the residual norm is
nondecreasing in $\alpha$. Since the penalty norm is positive on
$I_n^*$, the parameter map $\Phi_{n,K_n}$ is nondecreasing there.

Since $\Phi_{n,K_n}$ maps $I_n^*$ into itself,
all iterates remain in this interval.
If $\alpha^{(1)}\ge\alpha^{(0)}$, induction shows
that the sequence is nondecreasing; otherwise it is nonincreasing.
Being bounded, it converges to some $\alpha_*\in I_n^*$,
and continuity yields
$\alpha_*=\Phi_{n,K_n}(\alpha_*)$.
The finite-dimensional normal equations also imply
$f_{n,\alpha^{(j)}}\to f_{n,\alpha_*}$ as $j\to\infty$.

Finally, apply \eqref{qu-eq:source-transfer} to
$\alpha\in I_n^*$. By \eqref{qu-eq:prediction-rate} and the
choices of $u$ and $v$, the term containing the square brackets
is $o_{\mathbb P}((\log n)^{-\sigma})$ uniformly on $I_n^*$,
whereas
\[
(2M+1)U_n^{-\sigma}
=(2M+1)u^{-\sigma}(\log n)^{-\sigma}.
\]
Since \eqref{qu-eq:source-transfer} holds uniformly on $I_n^*$
with probability tending to one, it follows that, for a
sufficiently large constant $C>0$,
\[
\mathbb P\left(
\sup_{\alpha\in I_n^*}
\big\|f_{n,\alpha}-f^\dagger\big\|_{L^2(\Omega)}
\le C(\log n)^{-\sigma}
\right)\longrightarrow1.
\]
We have also shown that, with probability tending to one,
all iterates remain in $I_n^*$ and converge monotonically to
$\alpha_*\in I_n^*$. Thus, with probability tending to one,
the iteration converges and
\[
\big\|f_{n,\alpha_*}-f^\dagger\big\|_{L^2(\Omega)}
\le C(\log n)^{-\sigma}
\]
for a sufficiently large constant $C>0$, independent of $n$
and the noise realization. This completes the proof.
\end{proof}

\section{Numerical experiments}
\label{sec:experiments}

We first describe the spectral implementation, then present
numerical results on discrete sampling stability and the effects
of waiting time and noise on source reconstruction.

\subsection{Spectral implementation}

We consider the heat equation on the square domain
$\Omega=(0,1)^2$, with homogeneous Dirichlet boundary conditions,
zero initial value, and $\rho(t)=\mathbf 1_{[0,T_0]}(t)$, where
$T_0=0.5$. The observation time is $T=T_0+\tau_0$ for each chosen
source-free waiting time $\tau_0$.

The Dirichlet eigenfunctions and eigenvalues are
\[
\phi_{jk}(x_1,x_2)=2\sin(j\pi x_1)\sin(k\pi x_2),
\qquad \lambda_{jk}=\pi^2(j^2+k^2),\quad j,k\ge1.
\]

For $q>2$, let $\mathcal C_q\subset[0,1]$ be the two-branch Cantor
set obtained by retaining two subintervals of relative length $1/q$
at each stage. We use the observation set
\begin{equation*}
E_q=[\tfrac14,\tfrac34]\times
\bigl(\tfrac14+\tfrac12\mathcal C_q\bigr)
\Subset\Omega,
\qquad
\dim_H E_q=1+\frac{\log2}{\log q}>1.
\end{equation*}
The set $E_q$ has zero two-dimensional Lebesgue measure and positive
$\beta$-dimensional Hausdorff content for $1<\beta\le\dim_H E_q$. Thus, for $q>2$,
the set $E_q$ satisfies the theoretical observation-set assumption
$\mathcal H_\infty^\beta(E_q)>0$ for every
$1<\beta\le\dim_H E_q<2$.
For the numerical experiments, we fix $q=3$ and use
$E_3$ as the limiting observation set.

For the numerical experiments, we construct a finite candidate pool
using a depth-seven Cantor approximation. Let
$\mathcal M_3^{(7)}$ consist of the midpoints of the $2^7$ retained
intervals, and let $X_1^{(17)}$ consist of the $17$ equally spaced
points in $[1/4,3/4]$. The candidate pool is
\[
X_{\mathrm{cand}}^{(7)}
=\left\{\left(x_i,\frac14+\frac12 y\right):
x_i\in X_1^{(17)},\ y\in\mathcal M_3^{(7)}\right\},
\qquad |X_{\mathrm{cand}}^{(7)}|=17\cdot2^7=2176.
\]
Sensors are selected from this pool by farthest-point sampling.
The candidate list is ordered lexicographically, the greedy ordering
starts from its first point, and ties are resolved by the candidate
order. The same deterministic ordering is used for all sampling
budgets, so smaller sensor sets are nested in larger ones.
The finite pool supports the sampling budgets considered here;
asymptotic claims as $n\to\infty$ require refining the candidate pool.

\begin{figure}[!htb]
\centering
\subfigure[$n=50$]
{\includegraphics[width=0.46\textwidth]{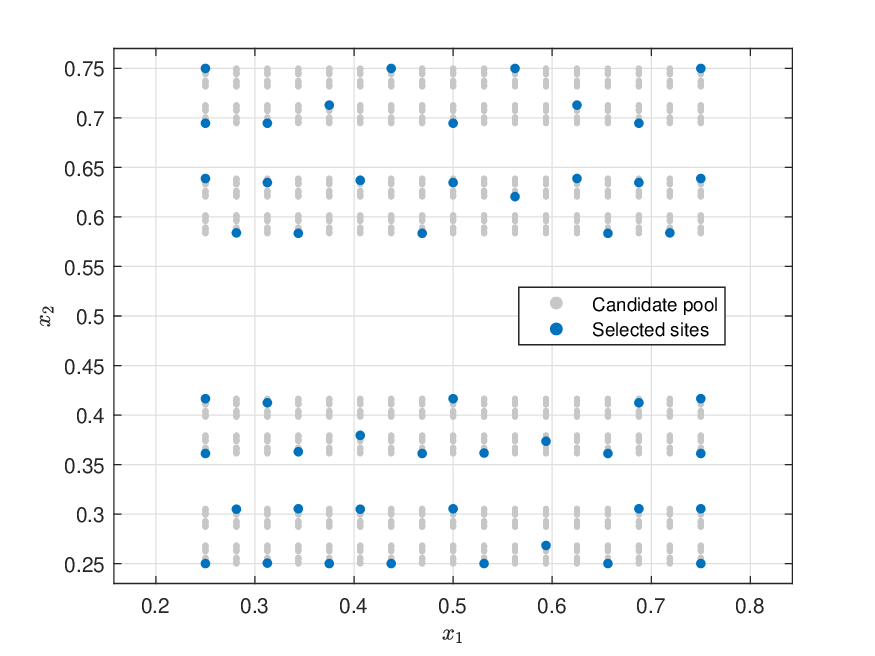}}
\hfill
\subfigure[$n=100$]{\includegraphics[width=0.46\textwidth]{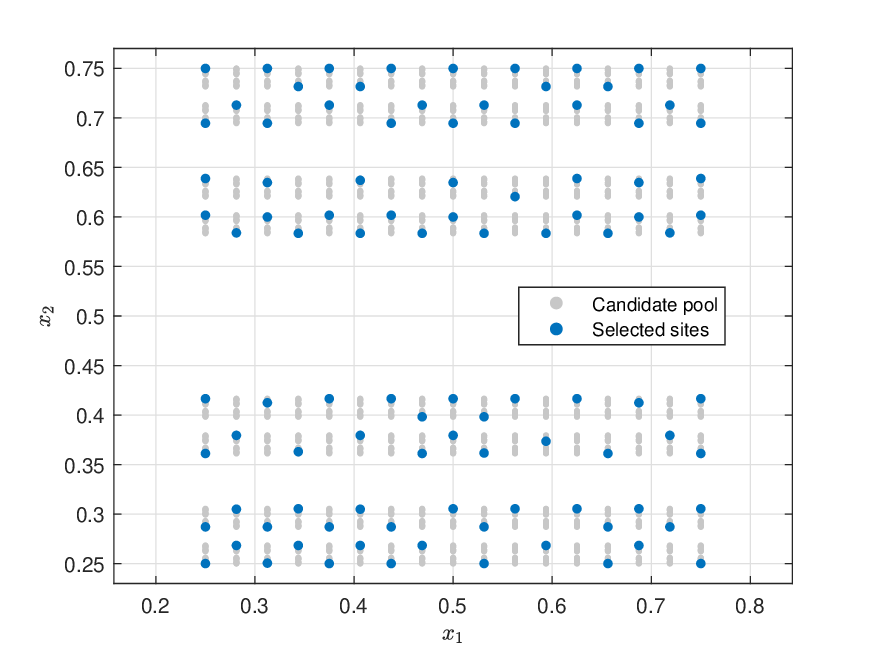}}

\subfigure[$n=200$]{\includegraphics[width=0.46\textwidth]{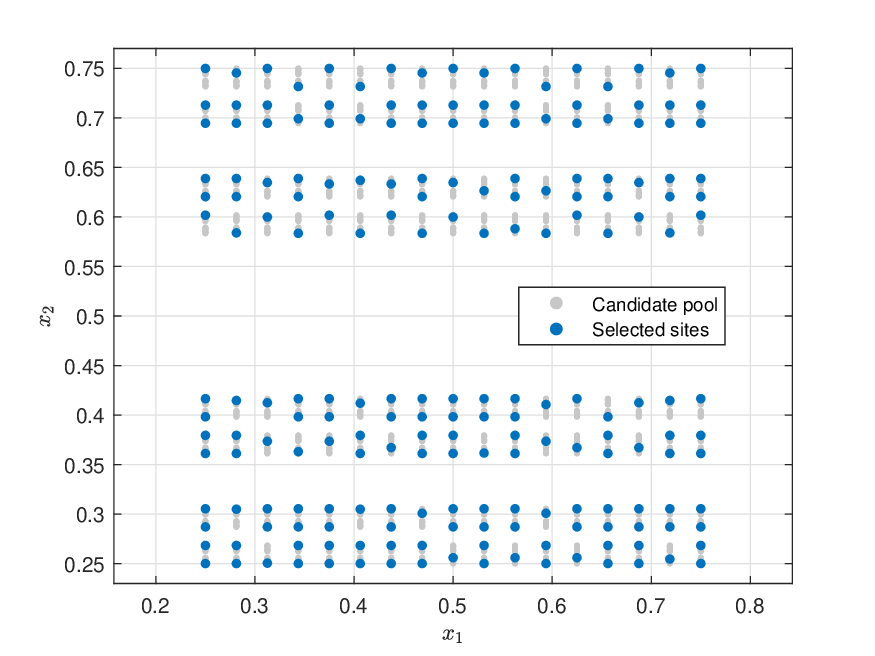}}
\hfill
\subfigure[$n=400$]
{\includegraphics[width=0.46\textwidth]{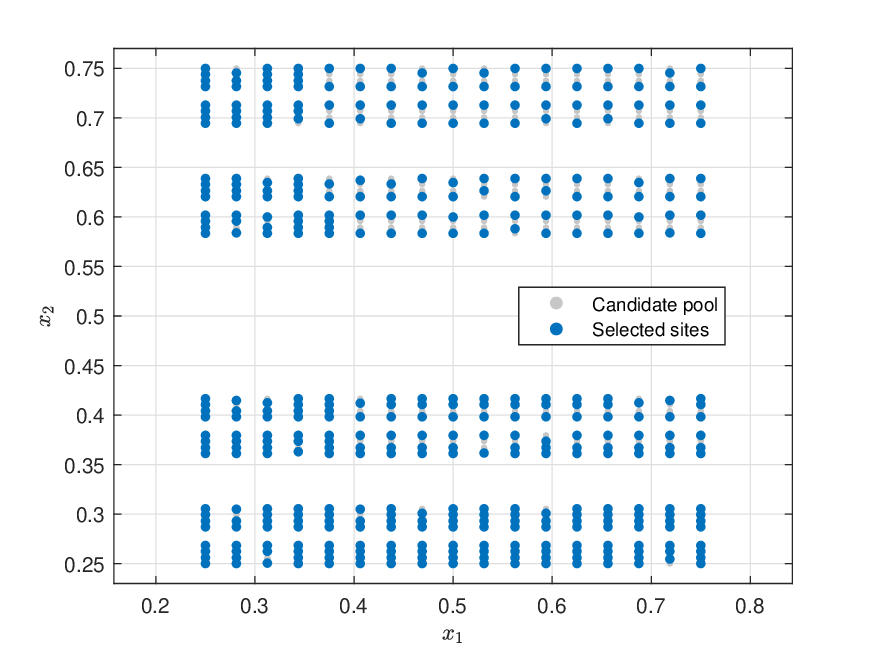}}
\caption{Sampling point distributions on the discrete approximation
of $E_3$ for different sampling budgets.}
\label{fig:sampling-E3}
\end{figure}

Figure~\ref{fig:sampling-E3} displays the sampling designs for
$n=50,100,200,$ and $400$. The point sets are nested,
because the same farthest-point ordering is used for all four sampling
budgets. As $n$ increases, more points are selected from the finite
candidate pool, giving progressively finer coverage of that pool while
keeping all sensors on the discrete approximation of $E_3$. Thus, the
sampling locations follow the Cantor-type observation geometry.

All reconstructions use the spectral Tikhonov functional
\eqref{eq:discrete-tikhonov} with $\sigma=1$ and the fixed-point map
\eqref{eq:clz-fixed-point}, where $q_\sigma=8/5$ and $p_0=4/5$. For each waiting time, we initialize the iteration with
\[
\alpha^{(0)}=c_0n^{-p_0},\qquad
c_0=\frac{|b_1|^2}{\lambda_1^{2\sigma}},
\]
where
\[
b_1=\frac{1-e^{-T_0\lambda_1}}{\lambda_1}
e^{-\tau_0\lambda_1},
\qquad \lambda_1=2\pi^2.
\]
The constant $c_0$ is determined by the known model parameters
and is independent of $n$ and the observed data.

At each iteration, we solve \eqref{eq:sampled-normal} with
$\alpha=\alpha^{(j)}$ and update the parameter using
\eqref{eq:clz-fixed-point}. The iteration stops when
\[
|\alpha^{(j+1)}-\alpha^{(j)}|\le10^{-12},
\]
with at most $500$ iterations.

We next consider two source profiles on $\Omega=(0,1)^2$. The first is the
smooth source
\[
f_{\mathrm{sm}}(x_1,x_2)=x_1(1-x_1)\sin(2\pi x_2).
\]
The second is a radial nonsmooth source centered at
$x_c=(1/2,1/2)$. Let $r=|x-x_c|$. We define
\[
f_{\mathrm{ns}}(x)=
\begin{cases}
1,&r\le R_0,\\[3pt]
(1-s^2)^{8/5},&R_0<r<R,\\[3pt]
0,&r\ge R,
\end{cases}
\qquad
s=\frac{r-R_0}{R-R_0},
\qquad R_0=0.2,\quad R=0.5.
\]
The first radial derivative matches continuously at both interfaces, and
its weak second derivatives belong to $L^2(\Omega)$. The profile also
vanishes on $\partial\Omega$; hence
\[
f_{\mathrm{ns}}\in H^2(\Omega)\cap H_0^1(\Omega)=D(A)
\]
for the square Dirichlet Laplacian. Thus this source is compatible with
the spectral prior used in the experiments, namely $\sigma=1$.

For either profile $f^\dagger\in\{f_{\mathrm{sm}},f_{\mathrm{ns}}\}$,
the terminal observation is represented by
\[
u(x,T)=\sum_{k=1}^{\infty}
b_k(\tau_0)(f^\dagger,\phi_k)\phi_k(x).
\]
For the chosen temporal profile, the multiplier is explicitly
\[
b_k(\tau_0)=\frac{1-e^{-T_0\lambda_k}}{\lambda_k}
e^{-\tau_0\lambda_k}.
\]

For both sources, the spectral coefficients are evaluated using
Gauss--Legendre quadrature with $600$ points in each coordinate.
The forward approximation retains all modes with $1\le j,k\le60$.

\subsection{Numerical reconstruction}

We conduct two experiments in this subsection. Experiment A examines discrete sampling
stability on a finite approximation of $E_3$, while Experiment B
investigates the effects of waiting time and noise on source
reconstruction.

\subsubsection*{Experiment A: Discrete sampling stability on a finite approximation of $E_3$}

We evaluate the discrete spectral sampling map on the finite candidate
set approximating $E_3$. For each $N\in\{6,10,15,20\}$, we use a
complete spectral cutoff
\[
R_6=10\pi^2,\qquad R_{10}=17\pi^2,\qquad
R_{15}=25\pi^2,\qquad R_{20}=32\pi^2,
\]
and define
\[
V_{R_N}:=\operatorname{span}\{\phi_k:\lambda_k\le R_N\}
=\operatorname{span}\{\phi_1,\ldots,\phi_N\}.
\]
The eigenfunctions are ordered by nondecreasing eigenvalue, with ties
resolved lexicographically by the index pair. Each cutoff includes every
mode at its terminal eigenvalue level, so $\dim V_{R_N}=N$ for all four
values of $N$. We independently vary the number $n\in\{50,100,200,400\}$
of sampling points and form
\begin{equation*}
\Phi_{N,n}=n^{-1/2}(\phi_k(x_{n,j}))_{\substack{1\le j\le n\\1\le k\le N}}.
\end{equation*}
For $v=\sum_{k=1}^N c_k\phi_k$, orthonormality gives
\[
\|v\|_{L^2(\Omega)}=\|c\|_2,
\qquad \|v\|_n=\|\Phi_{N,n}c\|_2.
\]
Whenever $\Phi_{N,n}$ has full column rank, its sharp numerical one-sided
sampling constant is
\begin{equation*}
C_{N,n}=\frac{1}{\sigma_{\min}(\Phi_{N,n})},
\end{equation*}
and we also compute
\[
\kappa(\Phi_{N,n})=\frac{\sigma_{\max}(\Phi_{N,n})}
{\sigma_{\min}(\Phi_{N,n})}.
\]
All tested configurations satisfy $n\ge N$. The constants are finite-dimensional diagnostics for the selected deterministic designs; they do not
by themselves verify the general sampling lemma or establish a uniform
asymptotic bound for farthest-point sampling.

\begin{table}[!htb]
\centering
\caption{Discrete sampling stability on a finite approximation of $E_3$:
computed one-sided sampling constants $C_{N,n}$ and condition numbers
$\kappa(\Phi_{N,n})$ for spectral dimension $N$ and sensor count $n$.}
\label{tab:experiment-A}
\setlength{\tabcolsep}{5pt}
\small
\begin{tabular}{c cc cc cc cc}
\toprule
\multirow{2}{*}{$N$}
& \multicolumn{2}{c}{$n=50$}
& \multicolumn{2}{c}{$n=100$}
& \multicolumn{2}{c}{$n=200$}
& \multicolumn{2}{c}{$n=400$} \\
\cmidrule(lr){2-3}\cmidrule(lr){4-5}
\cmidrule(lr){6-7}\cmidrule(lr){8-9}
& $C_{N,n}$ & $\kappa(\Phi_{N,n})$
& $C_{N,n}$ & $\kappa(\Phi_{N,n})$
& $C_{N,n}$ & $\kappa(\Phi_{N,n})$
& $C_{N,n}$ & $\kappa(\Phi_{N,n})$ \\
\midrule
$6$
& {1.1609} & {1.8292}
& {1.0853} & {1.7662}
& {1.1164} & {1.8687}
& {1.1503} & {1.9581} \\

$10$
& {1.6346} & {3.0603}
& {1.7302} & {3.2621}
& {1.7725} & {3.4829}
& {1.8174} & {3.5893} \\

$15$
& {2.5921} & {4.9834}
& {2.8308} & {5.5280}
& {3.0467} & {6.2027}
& {3.0683} & {6.2339} \\

$20$
& {3.9722} & {8.0589}
& {4.6056} & {9.1540}
& {4.9401} & {10.2574}
& {5.2893} & {11.2812} \\
\bottomrule
\end{tabular}
\end{table}

\begin{figure}[!htb]
\centering
\includegraphics[width=0.48\textwidth]{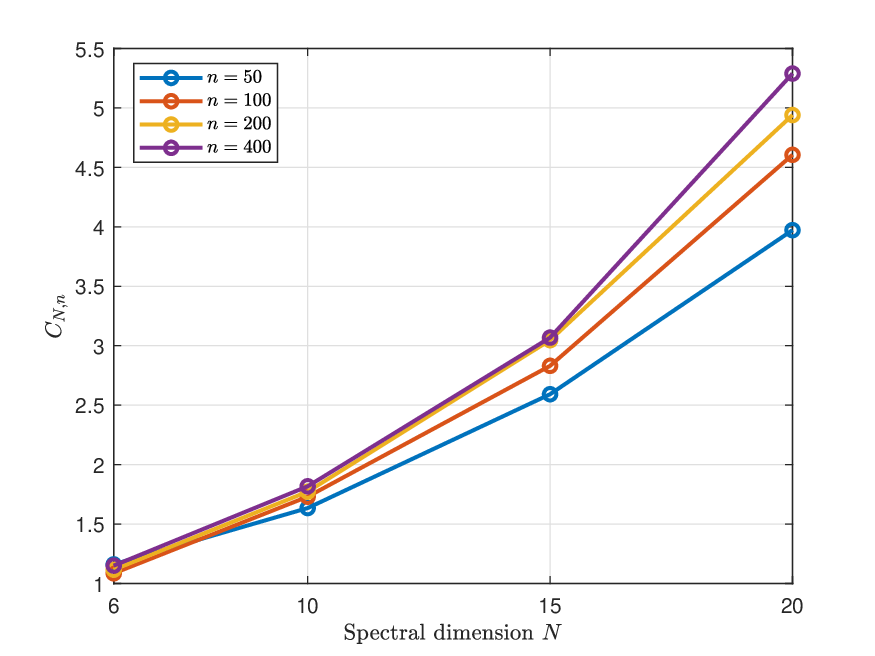}
\includegraphics[width=0.48\textwidth]{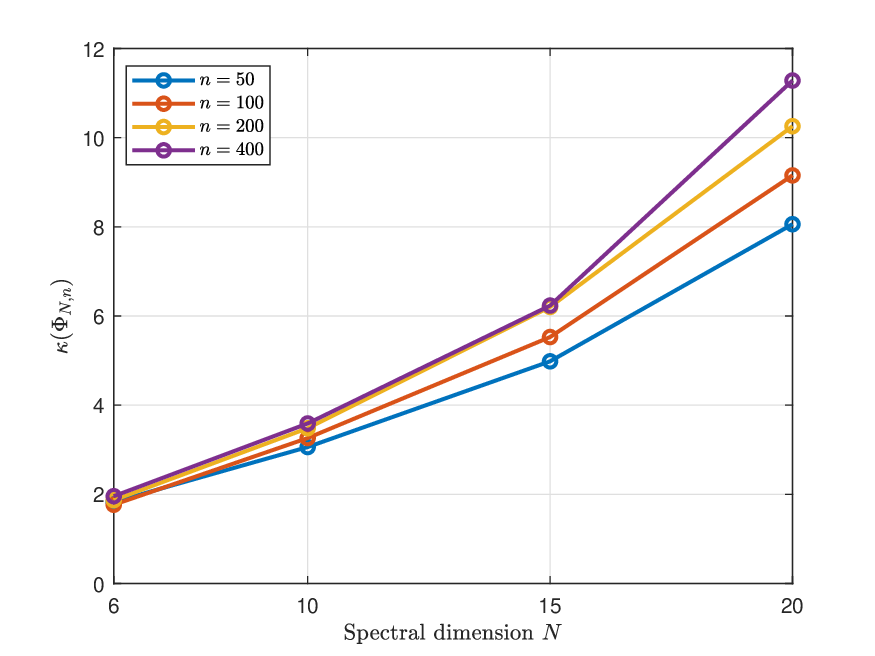}
\caption{$C_{N,n}$ and $\kappa(\Phi_{N,n})$ versus spectral dimension $N$,
with separate curves for $n=50,100,200,400$. Both vertical axes are logarithmic.}
\label{fig:experiment-A-constants}
\end{figure}

\begin{figure}[!htb]
\centering
\includegraphics[width=0.48\textwidth]{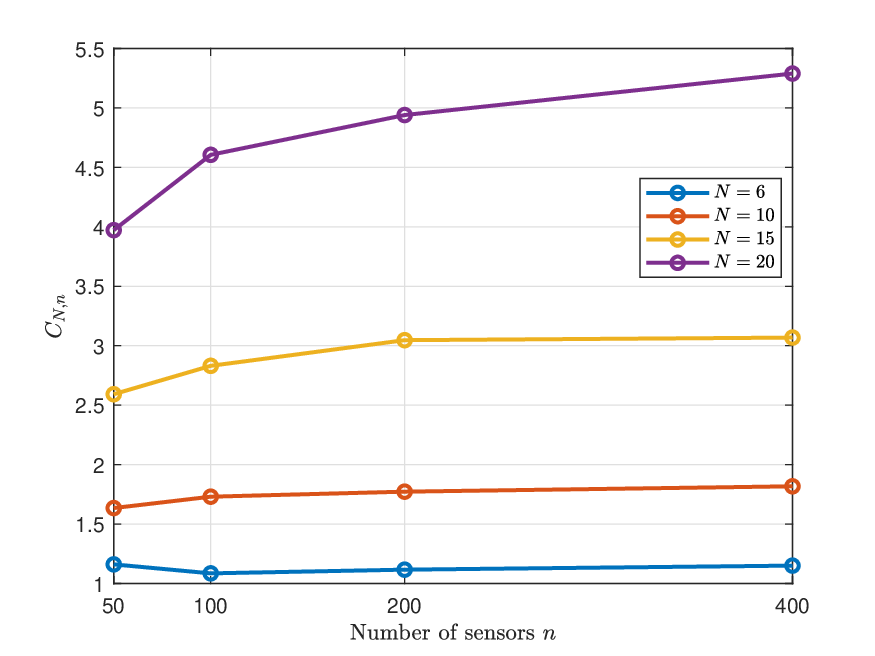}
\includegraphics[width=0.48\textwidth]{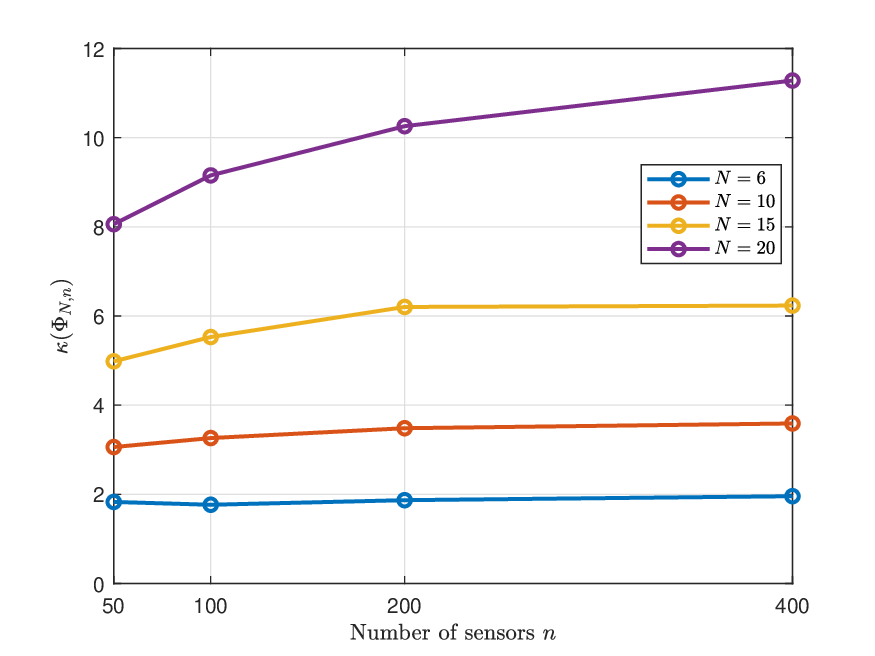}
\caption{$C_{N,n}$ and $\kappa(\Phi_{N,n})$ versus sensor count $n$,
with separate curves for $N=6,10,15,20$. Both vertical axes are logarithmic.}
\label{fig:experiment-A-vs-n}
\end{figure}

The computed values are reported in Table~\ref{tab:experiment-A}.
Figures~\ref{fig:experiment-A-constants} and
\ref{fig:experiment-A-vs-n} show their dependence on the spectral
dimension and sensor count, respectively. Both $C_{N,n}$ and
$\kappa(\Phi_{N,n})$ increase substantially with the spectral
dimension $N$. For example, at $n=400$, $C_{N,n}$ increases from
$1.1503$ for $N=6$ to $5.2893$ for $N=20$, while
$\kappa(\Phi_{N,n})$ increases from $1.9581$ to $11.2812$.
For fixed $N$, the dependence on the number of sensors is less
pronounced. The condition numbers generally increase with $n$ in
the tested designs, whereas $C_{6,n}$ varies only slightly and
nonmonotonically. Thus, increasing the number of sensors does not
necessarily improve these normalized finite-dimensional diagnostics;
the observed behavior also depends on the sampling geometry and
spectral dimension.

\subsubsection*{Experiment B: Waiting-time and noise effects}

We use $n=400$ sampling points selected from the finite candidate
approximation of $E_3$. The smooth source is reconstructed
in the first $10$ modes, while the nonsmooth source is reconstructed
in the first $15$ modes. For each source, synthetic observations are
computed from the source using the forward spectral approximation
described above. We consider source-free waiting times
$\tau_0\in\{0.01,0.02,0.04\}$ with $T_0=0.5$ fixed.

For each source and waiting time, we consider relative noise levels
$\varepsilon\in\{0,0.01,0.05,0.10\}$ and generate noisy observations
according to
\[
m_n^\delta
=m_n^{\mathrm{clean}}
+\varepsilon\|m_n^{\mathrm{clean}}\|_n\xi_n,
\]
where $m_n^{\mathrm{clean}}$ is the clean observation vector and
the entries of $\xi_n$ are independent standard Gaussian random
variables. For each fixed source, we reuse the same $50$ independent
Gaussian noise vectors across the positive noise levels and waiting
times, scaling each vector according to the formula above.
The $0\%$ case provides the noiseless baseline.
We report the median over the $50$ realizations of the relative
reconstruction error
\[
R_e=\frac{\|f_{\mathrm{rec}}-f^\dagger\|_{L^2(\Omega)}}
{\|f^\dagger\|_{L^2(\Omega)}},
\]
where $f^\dagger$ denotes the exact source and $f_{\mathrm{rec}}$
its reconstruction.

\begin{table}[!htb]
\centering
\caption{Median relative $L^2(\Omega)$ reconstruction errors
for the smooth and nonsmooth sources, based on $50$ Gaussian noise realizations.}
\label{tab:experiment-B}
\setlength{\tabcolsep}{6pt}
\small
\begin{tabular}{c*{8}{c}}
\toprule
& \multicolumn{4}{c}{Smooth source}
& \multicolumn{4}{c}{Nonsmooth source}\\
\cmidrule(lr){2-5}\cmidrule(lr){6-9}
$\tau_0$
& $0\%$ & $1\%$ & $5\%$ & $10\%$
& $0\%$ & $1\%$ & $5\%$ & $10\%$\\
\midrule
$0.01$
& {0.87\%}
& {1.53\%}
& {3.44\%}
& {4.74\%}
& {10.79\%}
& {16.94\%}
& {17.18\%}
& {17.25\%}\\

$0.02$
& {0.87\%}
& {1.99\%}
& {4.23\%}
& {4.76\%}
& {14.26\%}
& {17.07\%}
& {17.17\%}
& {17.23\%}\\

$0.04$
& {0.87\%}
& {3.82\%}
& {4.11\%}
& {4.28\%}
& {16.39\%}
& {17.10\%}
& {17.19\%}
& {17.26\%}\\
\bottomrule
\end{tabular}
\end{table}

\begin{figure}[!htb]
\centering
\subfigure[Exact source]{\includegraphics[width=0.48\textwidth]{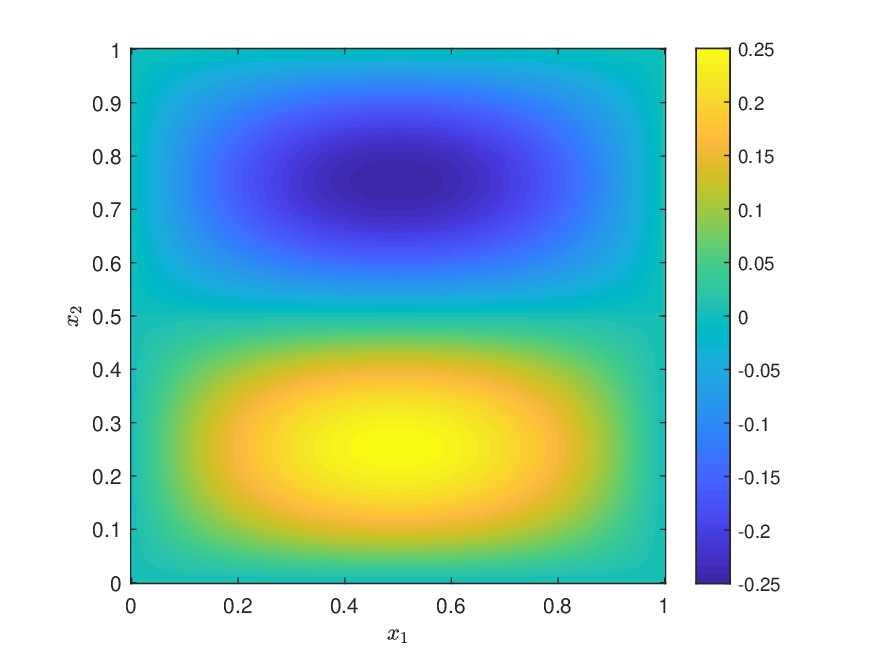}}
\subfigure[Reconstructed source]{\includegraphics[width=0.48\textwidth]{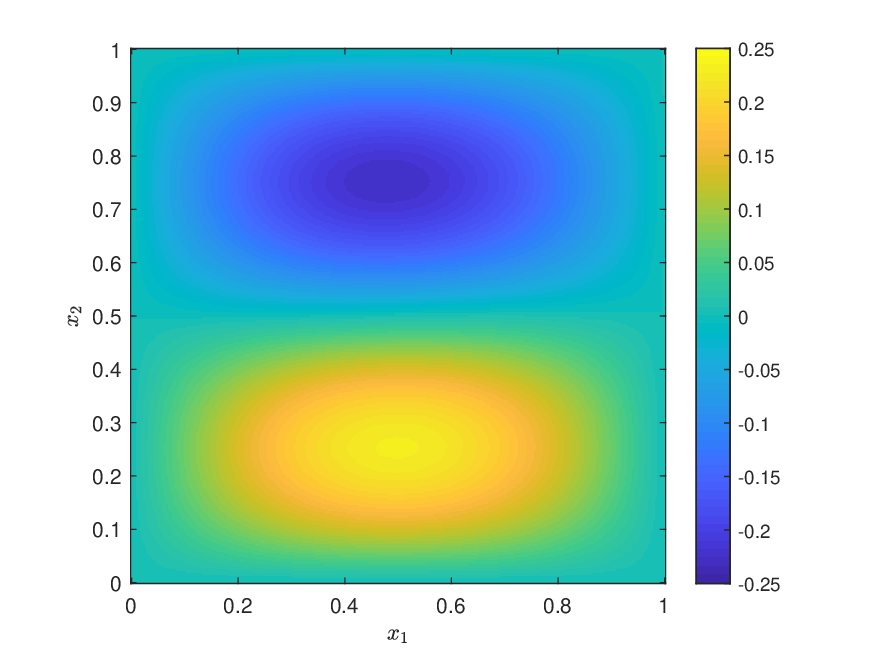}}
\caption{Reconstruction results for the smooth source at
$\tau_0=0.04$, $n=400$, and a $10\%$ relative noise level.}
\label{fig:experiment-B-smooth-tau004-noise10}
\end{figure}

\begin{figure}[!htb]
\centering
\subfigure[Exact source]{\includegraphics[width=0.48\textwidth]{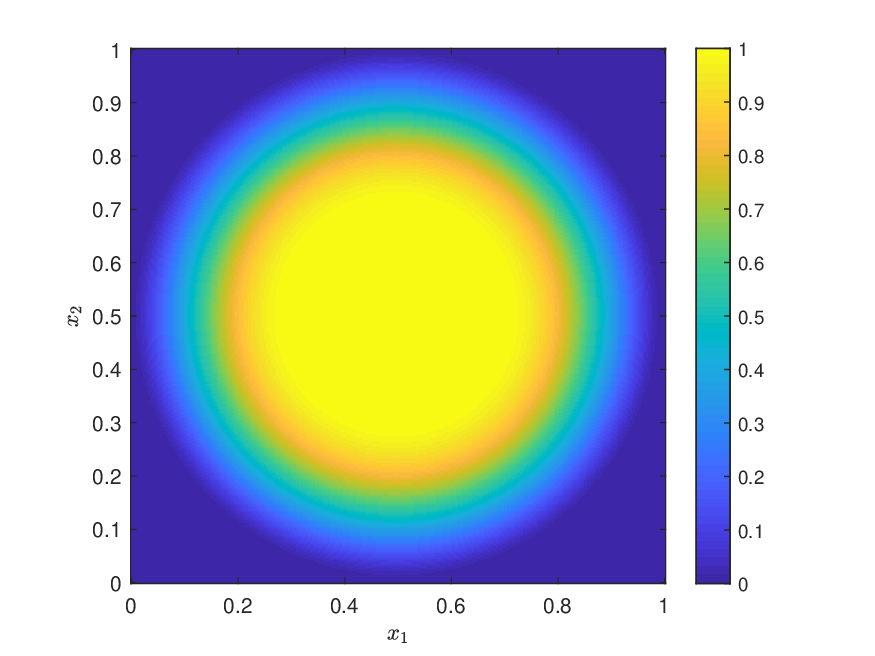}}
\subfigure[Reconstructed source]{\includegraphics[width=0.48\textwidth]{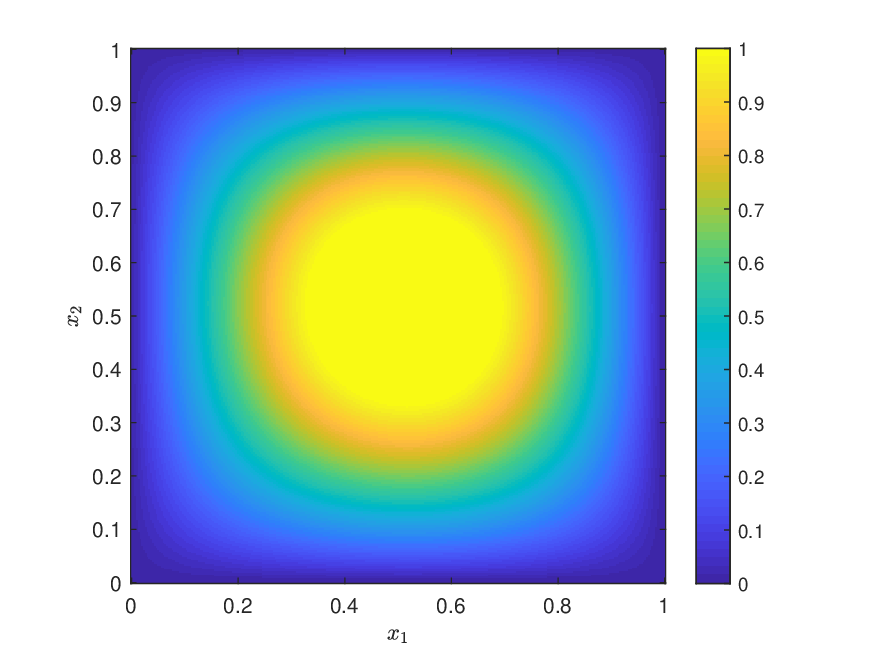}}
\caption{Reconstruction results for the nonsmooth source at
$\tau_0=0.04$, $n=400$, and a $10\%$ relative noise level.}
\label{fig:experiment-B-nonsmooth-tau004-noise10}
\end{figure}

The median reconstruction errors are summarized in
Table~\ref{tab:experiment-B}. Representative reconstructions at
$\tau_0=0.04$ and a $10\%$ relative noise level are shown in
Figures~\ref{fig:experiment-B-smooth-tau004-noise10} and
\ref{fig:experiment-B-nonsmooth-tau004-noise10} for the smooth and
nonsmooth sources, respectively. For the smooth source, reconstructed in the first $10$ modes, the
noiseless reconstruction error is approximately $0.87\%$ for all
three waiting times. At each fixed waiting time, the error increases
with the relative noise level. For example, at $\tau_0=0.01$, it
rises from $1.53\%$ at $1\%$ noise to $4.74\%$ at $10\%$ noise.
At $\tau_0=0.04$, the corresponding errors are $3.82\%$, $4.11\%$,
and $4.28\%$. The dependence on waiting time is not uniformly
monotone: at the $1\%$ noise level, the error increases from
$1.53\%$ to $3.82\%$ as $\tau_0$ increases, whereas at the $5\%$
and $10\%$ levels, the errors vary only slightly between
$\tau_0=0.02$ and $\tau_0=0.04$. Since the noise is scaled relative
to the clean-observation RMS at each waiting time, its absolute
amplitude decreases as the terminal signal weakens. The results
therefore reflect the combined effects of waiting time, noise, and
data-driven regularization rather than a uniformly increasing error
with waiting time.

For the nonsmooth source, reconstructed in the first $15$ modes,
the noiseless errors are $10.79\%$, $14.26\%$, and $16.39\%$ for $\tau_0=0.01$, $0.02$, and $0.04$, respectively. With positive noise,
the errors remain close to $17\%$ across all tested configurations.
At $\tau_0=0.01$ and $0.02$, introducing $1\%$ noise increases the
error substantially relative to the noiseless baseline, while further
increasing the noise level to $10\%$ produces only small changes.
At $\tau_0=0.04$, the error increases more moderately, from $16.39\%$ without noise to $17.10\%$ at $1\%$ noise, and then
remains near $17\%$. This behavior suggests that, for the nonsmooth
source, spectral truncation and regularization bias contribute
substantially to the reconstruction error, alongside the effect of
measurement noise.

\section{Concluding remarks}
\label{sec:conclusion}

The results of this work highlight the interplay between observation geometry and the temporal structure of the forcing in the recovery of a spatial source from terminal heat data. The positive $\beta$ dimensional Hausdorff content with $\beta>d-1$ provides the spectral observability needed to control low frequency components, while the waiting interval during which the forcing vanishes yields exponential decay of the terminal high frequency remainder. The comparison between the spectral observation cost $e^{C\sqrt{\Lambda}}$ and the waiting time decay $e^{-\tau_0\Lambda}$ is central to the conditional logarithmic stability estimate. The noncancellation condition on the terminal multipliers addresses a different difficulty: when the temporal profile changes sign, contributions from different parts of the forcing history may cancel at individual eigenfrequencies, compromising identifiability.

The discrete analysis connects this stability mechanism with reconstruction from scattered measurements. Under the stated quasi-uniform sampling and independent noise assumptions, the sampling inequality enables the continuous stability mechanism to be used in a finite dimensional setting, leading to a logarithmic convergence rate in probability for the data driven Tikhonov reconstruction. The numerical experiments illustrate how the waiting time, measurement noise, and source regularity affect reconstruction on a finite candidate approximation of a fixed Hausdorff set. These experiments complement the theoretical analysis by illustrating the behavior of the reconstruction method in representative settings.

Future work will focus on adaptive spectral cutoff selection and
the robustness of the data-driven parameter rule under more general
noise models. Another direction is to extend the scattered-point
sampling framework developed here to other inverse problems for partial differential equations, including the identification of unknown coefficients.

\section*{Acknowledgments}

This work is supported by the National Key Research and Development Program of China (Grant No.~2023YFB3002400). The second author also acknowledges support from the Ningbo Youth Leading Talent Project (No.~2024QL045), the National Natural Science Foundation of China (No.~12271277), and the Open Research Fund of the Key Laboratory of Nonlinear Analysis \& Applications (Central China Normal University), Ministry of Education, China.


\begin{thebibliography}{99}

\bibitem{BurqMoyano23}
N.~Burq and I.~Moyano,
Propagation of smallness and control for heat equations,
\emph{J. Eur. Math. Soc.}, 25 (2023), 1349--1377.

\bibitem{Cannon68}
J.~R. Cannon,
Determination of an unknown heat source from overspecified boundary data,
\emph{SIAM J. Numer. Anal.}, 5 (1968), 275--286.

\bibitem{CasagrandeDelSantoPrizzi22}
D.~Casagrande, D.~Del Santo, and M.~Prizzi,
Conditional stability up to the final time for
backward-parabolic equations with Log-Lipschitz coefficients,
\emph{J. Differential Equations}, 338 (2022),
518--550.

\bibitem{CenLiZhang25}
D.~Cen, Z.~Li, and W.~Zhang,
Scattered point measurement-based regularization for backward problems
for fractional wave equations,
\emph{J. Sci. Comput.}, 104 (2025), 58.

\bibitem{ChenZhangZou22}
Z.~Chen, W.~Zhang, and J.~Zou,
Stochastic convergence of regularized solutions and their finite
element approximations to inverse source problems,
\emph{SIAM J. Numer. Anal.}, 60 (2022), 751--780.

\bibitem{ChoulliYamamoto04}
M.~Choulli and M.~Yamamoto,
Conditional stability in determining a heat source,
\emph{J. Inverse Ill-Posed Probl.}, 12 (2004), 233--243.

\bibitem{Davies89}
E.~B. Davies, \emph{Heat Kernels and Spectral Theory}, Cambridge University Press, Cambridge, 1989.

\bibitem{EnglHankeNeubauer96}
H.~W.~Engl, M.~Hanke, and A.~Neubauer,
\emph{Regularization of Inverse Problems}, Kluwer, Dordrecht, 1996.

\bibitem{Falconer90}
K.~J. Falconer, \emph{Fractal Geometry: Mathematical Foundations and Applications}, Wiley, New York, 1990.

\bibitem{FuXiongQian07}
C.-L.~Fu, X.-T.~Xiong, and Z.~Qian,
Fourier regularization for a backward heat equation,
\emph{J. Math. Anal. Appl.}, 331 (2007),
472--480.

\bibitem{GilbargTrudinger01}
D.~Gilbarg and N.~S. Trudinger,
\emph{Elliptic Partial Differential Equations of Second Order}, 2nd ed., Springer, Berlin, 2001.

\bibitem{Green25}
A.~W. Green, K.~Le Balc'h, J.~Martin, and M.-A.~Orsoni,
On the dimension of observable sets for the heat equation,
\emph{Pure Appl. Anal.}, 7 (2025), 639--660.

\bibitem{Hansen10}
P.~C.~Hansen,
\emph{Discrete Inverse Problems: Insight and Algorithms},
SIAM, Philadelphia, 2010.

\bibitem{HuangWangWang26}
S.~Huang, G.~Wang, and M.~Wang,
Observability inequality, log-type Hausdorff content and heat equations,
\emph{Comm. Math. Phys.}, 407 (2026),  33.

\bibitem{ImanuvilovYamamoto98}
O.~Y. Imanuvilov and M.~Yamamoto,
Lipschitz stability in inverse parabolic problems by the Carleman estimate,
\emph{Inverse Problems}, 14 (1998), 1229--1245.

\bibitem{Isakov91}
V. Isakov,
Inverse parabolic problems with the final overdetermination,
\emph{Comm. Pure Appl. Math.}, 44 (1991), 185--209.

\bibitem{JaradatEtAl17}
A. Jaradat, F. Awawdeh, and M.~S.~M. Noorani,
Identification of time-dependent source terms and control parameters
in parabolic equations from overspecified boundary data,
\emph{J. Comput. Appl. Math.}, 313 (2017), 397--409.

\bibitem{KriegPozharskaUllrich26}
D.~Krieg, K.~Pozharska, M.~Ullrich, and T.~Ullrich,
Sampling projections in the uniform norm,
\emph{J. Math. Anal. Appl.}, 553 (2026),  129873.

\bibitem{Li06}
G. Li,
Data compatibility and conditional stability for an inverse source
problem in the heat equation,
\emph{Appl. Math. Comput.}, 173 (2006), 566--581.

\bibitem{LiYamamotoZou09}
J.~Li, M.~Yamamoto, and J.~Zou,
Conditional stability and numerical reconstruction of initial temperature,
\emph{Commun. Pure Appl. Anal.}, 8 (2009), 361--382.

\bibitem{LiLiZhou26}
Z.~Li, Z.~Li, and Z.~Zhou,
Inverse heat source problems from boundary flux and interior observations on sets of low Hausdorff dimension,
arXiv:2609.10071, 2026.

\bibitem{LimonovaMalykhinTemlyakov24}
I.~V.~Limonova, Yu.~V.~Malykhin, and V.~N.~Temlyakov,
One-sided discretization inequalities and sampling recovery,
\emph{Russian Math. Surveys}, 79 (2024), 515--545.

\bibitem{Liu02}
J.~Liu,
Numerical solution of forward and backward problem for
2-D heat conduction equation,
\emph{J. Comput. Appl. Math.}, 145 (2002),
459--482.

\bibitem{LogunovMalinnikova18}
A.~Logunov and E.~Malinnikova,
Quantitative propagation of smallness for solutions of elliptic equations,
in \emph{Proceedings of the International Congress of Mathematicians--Rio de Janeiro 2018}, Vol.~III, World Scientific, 2018, 2391--2411.

\bibitem{Malinnikova04}
E.~Malinnikova,
Propagation of smallness for solutions of generalized Cauchy--Riemann systems,
\emph{Proc. Edinb. Math. Soc.}, 47 (2004), 191--204.

\bibitem{Mattila95}
P.~Mattila, \emph{Geometry of Sets and Measures in Euclidean Spaces}, Cambridge University Press, Cambridge, 1995.

\bibitem{OrtegaPridhnani12}
J.~Ortega-Cerd\`a and B.~Pridhnani,
Beurling--Landau's density on compact manifolds,
\emph{J. Funct. Anal.}, 263 (2012), 2102--2140.

\bibitem{Ouhabaz05}
E.~M. Ouhabaz, \emph{Analysis of Heat Equations on Domains}, Princeton University Press, Princeton, 2005.

\bibitem{Pazy83}
A.~Pazy, \emph{Semigroups of Linear Operators and Applications to Partial Differential Equations}, Springer, New York, 1983.

\bibitem{PrilepkoTikhonov95}
A.~I. Prilepko and I.~V. Tikhonov,
Reconstruction of the nonhomogeneous term in an abstract evolution equation,
\emph{Izv. Math.}, 44 (1995), no.~2, 373--394.


\bibitem{Rao07}
K.~S. Rao,
Source estimation methods for atmospheric dispersion,
\emph{Atmospheric Environment}, 41 (2007), no.~33, 6964--6973.

\bibitem{SommarivaVianello09}
A.~Sommariva and M.~Vianello,
Computing approximate Fekete points by QR factorizations of Vandermonde matrices,
\emph{Comput. Math. Appl.}, 57 (2009), 1324--1336.

\bibitem{SunLiuLi17}
C.~Sun, Q.~Liu, and G.~Li,
Conditional well-posedness for an inverse source problem in the
diffusion equation using the variational adjoint method,
\emph{Adv. Math. Phys.}, 2017 (2017), 6801260.


\bibitem{SunZhangZhang25}
C.~Sun, W.~Zhang, and Z.~Zhang,
Quantitative estimates for a nonlinear inverse source problem in a
coupled diffusion equations with uncertain measurements,
arXiv:2504.19421, 2025.


\bibitem{SunZhang22}
C.~Sun and Z.~Zhang,
Uniqueness and numerical inversion in the time-domain fluorescence
diffuse optical tomography,
\emph{Inverse Problems}, 38 (2022), 104001.


\bibitem{WangAtmospheric21}
J. Wang, B. Wang, J. Liu, W. Cheng, and J. Zhang,
An inverse method to estimate the source term of atmospheric pollutant releases,
\emph{Atmospheric Environment}, 260 (2021), 118554.

\bibitem{WangZhangZhang231}
Z.~Wang, W.~Zhang, and Z.~Zhang,
A data-driven model reduction method for parabolic inverse source
problems and its convergence analysis,
\emph{J. Comput. Phys.}, 487 (2023), 112156.

\bibitem{WangZhangZhang23}
Z.~Wang, W.~Zhang, and Z.~Zhang,
Stochastic convergence of regularized solutions for backward heat conduction problems,
arXiv:2311.03623, 2023.

\bibitem{Yamamoto93}
M.~Yamamoto,
Conditional stability in determination of force terms of heat equations in a rectangle,
\emph{Math. Comput. Modelling}, 18 (1993), 79--88.

\bibitem{Yamamoto09}
M.~Yamamoto,
Carleman estimates for parabolic equations and applications,
\emph{Inverse Problems}, 25 (2009), 123013.

\bibitem{YuLi26}
K.~Yu and Z.~Li,
Logarithmic stability for recovering initial data of fractional heat equations from thin-set observations,
arXiv:2609.17253, 2026.

\end{thebibliography}
\end{document}